\documentclass{article}

\usepackage[leqno]{amsmath}
\usepackage{amssymb}
\usepackage{amsthm}
\usepackage{mathrsfs}
\usepackage[shortalphabetic]{amsrefs}
\usepackage[all]{xy}

\newcommand{\hdt}{{\dot{\mathrm{H}}^{1/2}}}
\newcommand{\hdtr}{{\dot{\mathrm{H}}^{1/2}(\mathbb{R}^3)}}
\newcommand{\R}{\mathbb{R}}
\newcommand{\ei}{\mathrm{e}^{it\Delta}}

\newtheorem{thm}{Theorem}[section]
\newtheorem{lem}[thm]{Lemma}
\newtheorem{pp}[thm]{Proposition}
\newtheorem{cor}[thm]{Corollary}

\theoremstyle{remark}
\newtheorem{rem}[thm]{Remark}

\theoremstyle{definition}
\newtheorem{dfn}[thm]{Definition}
\newtheorem{step}{Step}

\numberwithin{equation}{section}

\title{Scattering for $\hdt$ bounded solutions to the cubic, defocusing NLS in 3 
dimensions\footnote{2000 MSC number 35Q55}}

\author{Carlos E. Kenig\footnote{The first author was supported in part by NSF and the second one in part by CNRS. Part of this research was carried out during visits of the second author to the University of Chicago and IHES. Also, this research was supported in part by ANR ONDE NONLIN.}
\\Department of Mathematics\\University of Chicago\\Chicago, Il 60637\\USA\\cek@math.uchicago.edu \and 
Frank Merle\\Departement de Mathematiques\\Universite de Cergy--Pontoise\\Pontoise\\95302 Cergy--Pontoise\\ FRANCE\\Frank.Merle@math.u-cergy.fr}

\date{}

\begin{document}
\maketitle
\begin{abstract}
We show that if a solution of the defocusing cubic NLS in 3d remains bounded in the homogeneous Sobolev norm of order $1/2$ in its maximal interval of existence, then the interval is infinite and the solution scatters. No radial assumption is made.
\end{abstract}

\section{Introduction}

In this paper we continue our study of critical nonlinear dispersive problems, which we have developed in \cite{KM} and \cite{KM2}. In the present work we turn our attention to the defocusing, cubic NLS in three space dimensions, in the critical space $\hdt$. We then use a version of the concentration-compactness-rigidity method we introduced in \cite{KM}, to obtain the following:

\begin{thm}\label{thm11}
Suppose that $u$ is a solution \eqref{cp} with initial data $u_0\in\hdt(\R^3)$ (see section 2 for \eqref{cp}), and maximal interval of existence $I$ (Definition \ref{dfn27})). Assume that 
$ \sup_{0<t<T_+(u_0)}\lVert u(t)\rVert_{\hdt(\R^3)}=A<+\infty$. Then 
$T_+(u_0)=+\infty$ and $u$ must scatter at plus infinity, i.e. there exists $u_0^+$ so that
\begin{equation*}
\lim_{t \to + \infty}\lVert u(t)-\ei u_0^+\rVert_\hdt =0.
\end{equation*}
\end{thm}

Note that there in no radial assumption on $u_0$. Using the concentration\--compactness procedure (Propositions \ref{pp33} and \ref{pp34}) we show that if Theorem \ref{thm11} fails, there must exist a critical element, which enjoys a compactness property. Finally, in section 4 we establish a rigidity theorem, which shows that no such element can exist, using the well-known 
Lin--Strauss \cite{LS} estimate of Morawetz type. As a consequence of our result, the set of data $u_0$ as in Theorem \ref{thm11} is an open set in $\hdt(\R^3)$. Moreover if $u_0\in\hdt$ and $T_+(u_0)<\infty$, then $ \sup_{0<t<T_+(u_0)}\lVert u(t)\rVert_\hdt=+\infty$. Finally, an interesting open problem that we don't addres here is to show that for all data $u_0\in\hdt$, we must have $ \sup_{0<t<T_+(u_0)}\lVert u(t)\rVert_\hdt<+\infty$, for solutions of the cubic defocusing NLS in 3 dimensions.

We conclude this introduction by mentioning the work \cite{CKSTT}, in which  the authors were able to show scattering in \eqref{cp} for all data in $\mathrm{H}^s(\R^3)$, $s>4/5$. See also the references in \cite{CKSTT} for previous work in this problem.

\section{The Cauchy problem}

In this section we will review the Cauchy problem
\begin{equation}\label{cp}
\left\{\begin{array}{l}
i\partial_tu+\Delta u-|u|^2u=0\quad\quad(x,t)\in\R^3\times\R\\ \\
u\arrowvert_{t=0}=u_0\in\hdt(\R^3)\end{array}\right.
\end{equation}

This problem is $\hdt$ critical, because if $u(x,t)$ solves \eqref{cp}, so does 
$  u_{\lambda}(x)=
\frac{1}{\lambda}u\left(\frac{x}{\lambda},\frac{t}{\lambda^2}\right)$, with initial data
$  u_{0,\lambda}(x)=
\frac{1}{\lambda}u_0\left(\frac{x}{\lambda}\right)$ and 
$\lVert u_{0,\lambda}\rVert_\hdt=\lVert u_0\rVert_\hdt$. The nonlinearity is defocusing. The Cauchy problem theory (see \cite{CW}, \cite{KM}) depends on some previous results, which we now recall.
\begin{lem}[Strichartz estimates \cite{S}, \cite{KT}]\label{lem21}
We say that $(q,r)$ is admissible if $ \frac{2}{q}+\frac{3}{r}=\frac{3}{2}$ and
$2\leq q$, $r\leq\infty$. Then, if $2\leq r\leq 6$, $(m,n)$ is admissible and $2\leq m\leq 6$,
\begin{itemize}
\item[i)] $ \lVert \ei h\rVert_{L_t^q L_x^r}\leq C\lVert h\rVert_{L^2}$
\item[ii)] 
$ \left\lVert\int_{-\infty}^{+\infty}e^{i(t-\tau)\Delta}g(\cdot,\tau)d\tau\right\rVert_{L_t^q L_x^r}+\left\lVert\int_{0}^{t}e^{i(t-\tau)\Delta}g(\cdot,\tau)d\tau\right\rVert_{L_t^q L_x^r}\leq$\\
  $\leq C\lVert g\rVert_{L_t^{m'} L_x^{n'}}$ 
\item[iii)] $ \left\lVert\int_{-\infty}^{+\infty}e^{i(t-\tau)\Delta}g(\cdot,\tau)d\tau\right\rVert_{L_x^2}\leq
C\lVert g\rVert_{L_t^{m'} L_x^{n'}}$
\end{itemize}
\end{lem}

\begin{lem}[Sobolev embedding]\label{lem22}
For $v\in C_0^\infty(\R^4)$, we have
\begin{equation*} \lVert v\rVert_{L_t^5 L_x^5}\leq C\lVert D^{1/2}v\rVert_{L_t^5 L_x^{30/11}}.\end{equation*}
\end{lem}

\begin{lem}[Chain rule for fractional derivatives, \cite{KPV}]\label{lem23}
If $F\in C^2$, with $F(0)=0$, $F'(0)=0$, and $|F''(a+b)|\leq C\left\{|F''(a)|+|F''(b)|\right\}$, and
$|F'(a+b)|\leq C\left\{|F'(a)|+|F'(b)|\right\}$, we have, for $0<\alpha<1$,
\begin{equation*}\lVert D^\alpha F(u)\rVert_{L_x^p}\leq C\lVert  F'(u)\rVert_{L_x^{p_1}}
\lVert D^\alpha u\rVert_{L_x^{p_2}}\;, \quad\frac{1}{p}=\frac{1}{p_1}+\frac{1}{p_2},\end{equation*}
\begin{multline*}
\lVert D^\alpha[F(u)-F(v)]\rVert_{L_x^p}\leq\\\leq
C\left[\lVert F'(u)\rVert_{L_x^{p_1}}+\lVert F'(v)\rVert_{L_x^{p_1}}\right]\lVert D^\alpha(u-v)\rVert_{L_x^{p_2}}+\\ 
+C\left[\lVert F''(u)\rVert_{L_x^{r_1}}+\lVert F''(v)\rVert_{L_x^{r_1}}\right]\\\times
\left[\lVert D^\alpha u\rVert_{L_x^{r_2}}+\lVert D^\alpha v\rVert_{L_x^{r_2}}\right]
\lVert u-v\rVert_{L_x^{r_3}}\;,\end{multline*}
\begin{equation*}\frac{1}{p}=\frac{1}{r_1}+\frac{1}{r_2}+\frac{1}{r_3},\quad
\frac{1}{p}=\frac{1}{p_1}+\frac{1}{p_2}.\end{equation*}
\end{lem}

Let us define the $S(I)$, $W(I)$ norm for a time interval $I$ by
\begin{equation*}\lVert v\rVert_{S(I)}=\lVert v\rVert_{L_I^5L_x^5}\quad\quad\text{and}\quad\quad
\lVert v\rVert_{W(I)}=\lVert v\rVert_{L_I^5L_x^{30/11}}\;.\end{equation*}
Now, using Lemma \ref{lem21}, with $(q,r)=(5,30/11)$, $(m,n)=(5/2,30/7)$, $(m',n')=(5/3,30/23)$, we obtain, in a standard manner (see also \cite{CW}, \cite{KM} for similar proofs):
\begin{thm}[\cite{CW}, \cite{KM}]\label{thm24} 
Asume $u_0\in\hdt(\R^3)$, $t_0\in I$, $||u_0||_{\hdt(\R^3)}$ $\leq A$. Then there exists 
$\delta=\delta(A)$ such that if $||e^{i(t-t_0)\Delta}u_0||_{S(I)}<\delta$, there exists a unique solution $u$ to \eqref{cp} in $\R^3\times I$, with $u\in C(I;\hdt(\R^3))$,
\begin{gather*}
||D^{1/2}u||_{W(I)}+\sup_{t\in I}||D^{1/2}u(t)||_{L^2}\leq CA,\quad
||u||_{S(I)}\leq 2\delta.
\end{gather*}
Moreover, if $u_{0,k}\to u_0$ in $\hdtr$, the corresponding solutions $u_k\to u$ in $C(I;\hdtr)$.
\end{thm}

\begin{rem}\label{rem25}
There exists $\tilde\delta$ such that if $||u_0||_\hdtr\leq\tilde\delta$, the conclusion of Theorem
\ref{thm24} holds. This is because of Lemmas \ref{lem21}, \ref{lem22}.
\end{rem}

\begin{rem}\label{rem26}
Given $u_0\in\hdt$, there exists $(0\in)I$ such that the hypothesis of Theorem \ref{thm24} is verified on $I$. This is clear from Lemmas \ref{lem21}, \ref{lem22}.
\end{rem}

\begin{dfn}\label{dfn27}
Let $t_0\in I$. We say that $u\in C(I,\hdtr)\cap\{D^{1/2}u\in W(I)\}$ is a solution of \eqref{cp} if
\begin{equation*}u\arrowvert_{t_0}=u_0\quad\quad\text{and}\quad\quad u(t)=e^{i(t-t_0)\Delta}u_0+
\int_{t_0}^t e^{i(t-t')\Delta}f(u)dt'\end{equation*}
with $f(u)=-|u|^2u$.

It is easy to see that solutions of \eqref{cp} are unique (see 2.10 in \cite{KM}, for example). This allows us to define a maximal interval $I(u_0)$, where the solution is defined. $I(u_0)=(t_0-T_-(u_0), t_0+T_+(u_0))$ and if $I'\subset\subset I(u_0)$, $u$ solves \eqref{cp} in $\R^3\times I'$, so that $u\in C(I';\hdtr)$, 
$D^{1/2}u\in W(I')$, $u\in S(I')$.
\end{dfn}

\begin{lem}[Standard finite blow-up criterion]\label{lem28}
If $T_+(u_0)<+\infty$, then 
\begin{equation*}||u||_{S([t_0,t_0+T_+(u_0)))}=+\infty.\end{equation*}
A corresponding result holds for $T_-(u_0)$.
\end{lem}
See \cite{KM}, Lemma 2.11, for instance, for a similar proof.

\begin{rem}[See Remark 2.15 in \cite{KM}]\label{rem29}
If $u$ is a solution of \eqref{cp} in $\R^3\times I$, $I=[a,+\infty)$ (or $I=(-\infty, a]$), there exists 
$u_+\in\hdt$ such that
\begin{equation*}\lim_{t\uparrow+\infty}||u(t)-\ei u_+||_\hdt=0\end{equation*}
This is a consequence of the fact that $||u||_{S(I)}<\infty$.
\end{rem}

In the next section we will also need the notion of nonlinear profile.
\begin{dfn}\label{dfn210}
Let $v_0\in\hdt$, $v(t)=\ei v_0$ and let $\{t_n\}$ be a sequence, with 
$ \lim_{n\to\infty}t_n=\overline t\in[-\infty,+\infty]$. We say that $u(x,t)$ is a nonlinear profile associated with $(v_0,\{t_n\})$ if there exists an interval $I$, with $\overline t\in I$
(if $\overline t=\pm\infty$, $I=[a,+\infty)$ or $I=(-\infty, a]$) such that $u$ is a solution of \eqref{cp}
in $I$ and
\begin{equation*}\lim_{n\to\infty}||u(\cdot,t_n)-v(\cdot,t_n)||_\hdt=0.\end{equation*}
\end{dfn}

\begin{rem}\label{rem211}
There always exists a unique nonlinear profile associated to $(v_0,$ $\{t_n\})$. (For a proof, see the analogous one in Remark 2.13, \cite{KM}). We can hence define a maximal interval $I$ of existence for the nonlinear profile associated to $(v_0,\{t_n\})$.
\end{rem}

We conclude this section with a perturbation theorem that is fundamental in the sequel. For a proof of this theorem, see \cite{HR}, Proposition 2.3.

\begin{thm}[Perturbation theorem]\label{thm212}
Let $I\subset\R$ be a time interval and let $t_0\in I$. Let $\tilde u$ be defined on $\R^3\times I$ such that 
$ \sup_{t\in I}||\tilde u(t)||_\hdt\leq A$, $||\tilde u||_{S(I)}\leq M$, 
$||D^{1/2}\tilde u||_{W(I)}<+\infty$, for some constants $M,A>0$. Assume that
\begin{equation*}i\partial_t\tilde u+\Delta\tilde u-|\tilde u|^2\tilde u=e\quad\quad (x,t)\in\R^3\times I\end{equation*}
(in the sense of the appropriate integral equation) and let $u_0\in\hdt$ be such that 
$||u_0-\tilde u(t_0)||_\hdt\leq A'$. 

Then, there exists $\epsilon_0=\epsilon_0(M,A,A')>0$ such that if $0<\epsilon\leq\epsilon_0$ and
\begin{equation*}||D^{1/2}e||_{L_I^{5/3}L_x^{30/23}}\leq\epsilon,\quad\quad
||e^{i(t-t_0)\Delta}[u_0-\tilde u(t_0)]||_{S(I)}\leq\epsilon,\end{equation*}
 there exists a unique solution $u$ of \eqref{cp} on $\R^3\times I$, such that $u\arrowvert_{t=t_0}=u_0$ and
\begin{gather*}
||u||_{S(I)}\leq C(A,A',M),\quad\quad ||u-\tilde u||_{S(I)}\leq C(A,A',M)(\epsilon+\epsilon'),\\
\sup_{t\in I}||u(t)-\tilde u(t)||_\hdt+||D^{1/2}(u-\tilde u)||_{W(I)}\leq C(A,A',M)(A'+\epsilon+\epsilon')
\end{gather*}
where $\epsilon'=\epsilon^\beta$, for some $\beta>0$.
\end{thm}

\begin{rem}\label{rem213}
Theorem \ref{thm212} also yields the following continuity fact: let $\tilde u_0\in\hdt$, 
$||\tilde u_0||_\hdt\leq A$, and $\tilde u$ be a solution of \eqref{cp}, $t_0=0$, with maximal interval of existence $(-T_-(u_0),$ $T_+(u_0))$. Let $u_{0,n}\to\tilde u_0$ in $\hdt$, and let $u_n$be the corresponding solution of \eqref{cp}, with maximal interval of existence $(-T_-(u_{0,n}), T_+(u_{0,n}))$. Then
$-T_-(u_0)\geq-T_-(u_{0,n})$, for all $n$ large and $T_+(u_0)\leq T_+(u_{0,n})$, for all $n$ large. Moreover, for each $t\in(-T_-(u_0), T_+(u_0))$, $u_n(t)\to\tilde u(t)$ in $\hdt$. (See Remark 2.17 in \cite{KM}.)
\end{rem}

\begin{rem}\label{rem214}
Theorem \ref{thm212} also yields the following: Let $K\subset\hdt$ be such that $\overline K$ is compact. Then
$\exists\; T_K^+, T_K^-$ such that, for all $u_0\in K$ we have $T_+(u_0)> T_K^+$, $T_-(u_0)> T_K^-$. Moreover, the family $\{u(t): t\in[-T_K^-, T_K^+], u_0\in K\}$ has compact closure in $C([-T_K^-, T_K^+];\hdt)$
and hence is equicontinuous and bounded.
\end{rem}

\section{Concentration-compactness procedure}

In this section we will carry out the concentration-compactness argument which, combined with the rigidity theorem in the next section, will yield our result.This procedure is similar to the one the authors developed in \cite{KM}, \cite{KM2}, but with one important distinction. Here we do not use any conservation law, which makes the proof necessarily more delicate. The argument we use here should have further applications. For instance, it can be applied to yield a proof of Corollary 5.16 in \cite{KM} and of Corollary 7.4 in \cite{KM2}.

\begin{dfn}\label{dfn31}
For $A>0$, 
\begin{multline*}
B(A)=\{u_0\in\hdt:\text{if }
u\text{ is the solution of \eqref{cp}, equal}\\\text{ to }u_0\text{ at }t=0,
\text{ then }\sup_{t\in\left[0,T_+(u_0)\right)}||D^{1/2}u(t)||_{L^2}\leq A\}.
\end{multline*}
$  B(\infty)=\bigcup_{A>0}B(A)$.
\end{dfn}

\begin{dfn}\label{dfn32}
We say that $SC(A)$ holds if for each $u_0\in B(A)$, $T_+(u_0)=+\infty$ and $||u||_{S(0,+\infty)}<\infty$.
We also say that $SC(A;u_0)$ holds if $u_0\in B(A)$, $T_+(u_0)=+\infty$ and $||u||_{S(0,+\infty)}<\infty$.
\end{dfn}

By Theorem \ref{thm24} and Lemma \ref{lem21}, Lemma \ref{lem22}, we see that, for $\tilde\delta_0$ small enough, we have that if $||u_0||_\hdt\leq\tilde\delta_0$, then $SC(C\tilde\delta_0,u_0)$ holds. By a similar argument, there exists $A_0>0$ small enough, such that $SC(A_0)$ holds. Our main result, Theorem \ref{thm11} is equivalent to the statement that $SC(A)$ holds for each $A>0$. Thus, if Theorem \ref{thm11} fails, there exists a critical value $A_C$ with the property that, if $A<A_C$, $SC(A)$ holds, but if $A>A_C$, $SC(A)$ fails. Moreover, $A_C>A_0$. The concentration-compactness procedure consists in establishing the following key propositions:

\begin{pp}\label{pp33}
There exists $u_{0,C}$ such that $SC(A_C;u_{0,C})$ fails.
\end{pp}

\begin{pp}\label{pp34}
If $u_{0,C}$ is as in Proposition \ref{pp33}, then there exist $x(t)\in\R^3$, $\lambda(t)\in\R^+$, for 
$t\in[0,T_+(u_0))$, such that 
\begin{equation*}  K=\left\{v(x,t)=\frac{1}{\lambda(t)}
u_C\left(\frac{x-x(t)}{\lambda(t)}, t\right)\right\},
\end{equation*} 
$t\in[0,T_+(u_{0,C}))$, has the property that $\overline K$ is compact in $\hdt$. Here $u_C$ is the solution of \eqref{cp} with data $u_{0,C}$ at $t=0$.
\end{pp}

The key tool in in the proof of Proposition \ref{pp33} and \ref{pp34} is the following ``profile decomposition''.

\begin{lem}\label{lem35}
Given $\{v_{0,n}\}\subseteq\hdt$, with $||v_{0,n}||_\hdt\leq A$, there exist a sequence 
$\{V_{0,j}\}_{j=1}^\infty\subseteq\hdt$, a subsequence of $\{v_{0,n}\}$, and a sequence of triples $(\lambda_{j,n};x_{j,n};t_{j,n})\in\R^+\times\R^3\times\R$, which are ``orthogonal'' i.e.
\begin{equation*}
\frac{\lambda_{j,n}}{\lambda_{j',n}}+\frac{\lambda_{j',n}}{\lambda_{j,n}}+\frac{|t_{j,n}-t_{j',n}|}{\lambda_{j,n}^2}
+\frac{|x_{j,n}-x_{j',n}|}{\lambda_{j,n}}\rightarrow\infty
\end{equation*}
as $n\to\infty$, for $j\neq j'$, such that, for each $J\geq1$, we have
\begin{itemize}
\item[(i)] 
$  v_{0,n}=\sum_{j=1}^J\frac{1}{\lambda_{j,n}}V_j^l\left(\frac{\cdot-x_{j,n}}{\lambda_{j,n}},
-\frac{t_{j,n}}{\lambda_{j,n}^2}\right)+w_n^J$\\
where $V_j^l(x,t)=\ei V_{0,j}$ ($l$ stands for linear solution) and
\item[(ii)] $ \varlimsup_{n\to\infty}||\ei w_n^J||_{S(-\infty,+\infty)}\xrightarrow[J\to\infty]{}0$, and
\item[(iii)] For each $J\geq1$ we have 
\begin{equation*} 
||v_{0,n}||_\hdt^2=\sum_{j=1}^J||V_{0,j}||_\hdt^2+||w_n^J||_\hdt^2+\epsilon^J(n)
\end{equation*}
where $\epsilon^J(n)\xrightarrow[n\to\infty]{}0$.
\end{itemize}
\end{lem} 

The proof of Lemma \ref{lem35} is completely analogous to the one of Theorem 1.6 in Keraani \cite{K} and will be omitted.

Another ingredient in the proof of Propositions \ref{pp33} and \ref{pp34} is the following:
\begin{lem}\label{lem36}
Assume that $||h_n||_\hdt\leq A$ and that $||\ei h_n||_{S(0,+\infty)}\xrightarrow[n\to\infty]{}0$. Then 
$\xymatrix{ D^{1/2}h_n\ar@^{->}[r]_{\;\;\;\;\;\;n\to\infty}&0}$
 weakly in $L^2(\R^3)$.
\end{lem}
\begin{proof}
Fix $g\in S(0,+\infty)^\ast=L^{5/4}_{(0,+\infty)}L^{5/4}_x$, with $||g||_{S(0,+\infty)^\ast}\leq 1$ and 
$  g(x,t)=\sum_{\alpha=1}^M g_\alpha(t)f_\alpha(x)$, where $g_\alpha\in C_0^\infty(0,+\infty)$, $f_\alpha\in\mathscr{S}(\R^3)$,
$\hat f_\alpha(\xi)=0$ for $|\xi|$ small.Such $g$ are dense in the unit ball of $L^{5/4}_{(0,+\infty)}L^{5/4}_x$. Let now $v$ be such that 
$D^{1/2}h_{n_j}\rightharpoonup v$ weakly in $L^2(\R^3)$, for some subsequence $\{n_j\}$ and let $h=D^{-1/2}v$. Then,
\begin{multline*}
\int\!\!\!\int \ei h\overline g=\\=\sum_{\alpha=1}^M\int\!\!\!\int h\overline{g_\alpha(t)}\;\overline{e^{-it\Delta}f_\alpha}=
\sum_{\alpha=1}^M\int\!\!\!\int v\overline{g_\alpha(t)}\;\overline{e^{-it\Delta}(D^{-1/2}f_\alpha)}=\\
=\lim_{j\to\infty}\sum_{\alpha=1}^M\int D^{1/2}h_{n_j}\overline{\left(\int e^{-it\Delta}(D^{-1/2}f_\alpha)g_\alpha(t) dt\right)}dx=\\
=\lim_{j\to\infty}\int\!\!\!\int \ei h_{n_j}\overline{g}.
\end{multline*}
Hence, for such $g$ we have 
\begin{equation*}
\left|\int\!\!\!\int \ei h\overline g\right|\leq\varlimsup_{j}||\ei h_{n_j}||_{S(0,+\infty)}=0.
\end{equation*}
This shows that $||\ei h||_{S(0,+\infty)}=0$, $h\in\hdt$. From this it is easy to conclude that $h\equiv0$, so that $v\equiv0$.
 $\square$ \end{proof}

\begin{proof}[Proof of Proposition \ref{pp33}]
Let us find $A_n\downarrow A_C$, and $u_{0,n}\in\hdt$, with corresponding solution $u_n$ with 
$$\sup_{0\leq t\leq T_+(u_{0,n})}||u_n(t)||_\hdt\leq A_n$$ and $ ||u_m||_{S(0,T_+(u_{0,n}))}=+\infty$.
(Here we use, when $T_+(u_{0,n})<\infty$, Lemma \ref{lem28}). We will now use the ``profile decomposition'', Lemma \ref{lem35}, for the sequence  $\{u_{0,n}\}$, (under the assumption $A_C<\infty$), so that $A_n\leq 2A_C$ for all $n$. We will pass to a subsequence as in Lemma \ref{lem35} and assume, passing to a further subsequence, that $s_{j,n}=-t_{j,n}/\lambda_{j,n}^2\xrightarrow[n]{}s_j\in[-\infty,+\infty]$ for each $j=1,2, \ldots$, and for each $J=1, \ldots$, we have 
\begin{equation*}
\varlimsup_{n\to\infty}||\ei w_n^J||_S=\lim_{n\to\infty}||\ei w_n^J||.
\end{equation*} 
We will also let $U_j$ be the non-linear profile associated with $(V_{0,j},\{s_{j,n}\})$, (Definition \ref{dfn210}) and we will let 
\begin{equation*}
 \tilde U_{j,n}(x,t)=\frac{1}{\lambda_{j,n}}U_j\left(\frac{x-x_{j,n}}{\lambda_{j,n}},\frac{t}{\lambda_{j,n}^2}+s_{j,n}\right),
\end{equation*}
 which is also a solution of \ref{cp}. The proof will now be accomplished in a number of steps.
\begin{step} There exists $J_0>0$ such that, for $j>J_0$, we have $T_\pm(U_j)=+\infty$ and
\begin{equation*}
\sup_{t\in(-\infty,+\infty)}\!\!\!\!||U_j(t)||_\hdt+||U_j||_{S(-\infty,+\infty)}+||D^{1/2}U_j||_{W(-\infty,+\infty)}\leq C||V_{0,j}||_\hdt
\end{equation*}
\end{step}

To establish this step, note that, from (iii), by choosing $n$ large, for any $J\geq1$ we have
\begin{equation}\label{eq37}
||w_n^J||^2_\hdt+\sum_{j=1}^J||V_{0,j}||^2_\hdt\leq||v_{0,n}||^2_\hdt+A_C^2\leq5A_C^2
\end{equation}
Thus, for $J_0$ large, $j\geq J_0$, we have $||V_{0,j}||_\hdt\leq\tilde\delta$, with $\tilde\delta$ as in Remark \ref{rem25}, so that
$||\ei V_{0,j}||_{S(-\infty,+\infty)}\leq\delta$. From the construction of the non-linear profile $U_j$, it now follows that 
$||U_j||_{S(-\infty,+\infty)}\leq2\delta$ and 
\begin{equation*}
\sup_{t\in(-\infty,+\infty)}\!\!\!\!||U_j(t)||_\hdt+||D^{1/2}U_j||_{W(-\infty,+\infty)}\leq C||V_{0,j}||_\hdt
\end{equation*}
which establishes this step.

\begin{step} It cannot happen that for all $1\leq j\leq J_0$, $n$ large, we have $$||U_j||_{S(s_{j,n}, T_+(U_j)))}<\infty.$$
\end{step}

If not, the proof of Lemma \ref{lem28} (see for instance \cite{KM}, Lemma 2.1) gives both that $T_+(U_j)=+\infty$ and that
\begin{equation*}
\sup_{t\in(s_{j,n},+\infty)}\!\!\!\!||U_j(t)||_\hdt+||D^{1/2}U_j||_{W(s_{j,n},+\infty)}< \infty,
\end{equation*}
for $1\leq j\leq J_0$, so that, combining this with Step 1, we obtain
\begin{multline}\label{eq38}
\sum_{j=1}^\infty\sup_{t\in(s_{j,n},+\infty)}\!\!\!\!||U_j(t)||^2_\hdt+\\+||U_j||^2_{S(s_{j,n},+\infty))}
||D^{1/2}U_j||^2_{W(s_{j,n},+\infty)}\leq C_0.
\end{multline}
For $\epsilon_0>0$ to be chosen choose $J(\epsilon_0)$ so that, for $n$ large 
\begin{equation*}
||\ei w_n^{J(\epsilon_0)}||_{S(-\infty,+\infty)}\leq\epsilon_0. 
\end{equation*}
Let $  H_{n,\epsilon_0}=\sum_{j=1}^{J(\epsilon_0)}\tilde U_{j,n}(x,t)$. We next show that \eqref{eq38} and orthogonality give that
\begin{equation}\label{eq39}
||H_{n,\epsilon_0}||_{S(0,+\infty)}+\sup_{t\in(0,+\infty)}\!\!\!\!||H_{n,\epsilon_0}||_\hdt+||D^{1/2}H_{n,\epsilon_0}||_{W(0,+\infty)}\leq\tilde C_0,
\end{equation}
for $n\geq n(\epsilon_0)$.

The proof of the bound for $||H_{n,\epsilon_0}||_{S(0,+\infty)}$ is similar to the one given in \cite{KM}, pages 663--664. We next show the other two bounds: recall that $H_{n,\epsilon_0}$ verifies
\begin{equation*}
\left\{
\begin{array}{l} 
i\partial_tH_{n,\epsilon_0}+\Delta H_{n,\epsilon_0}=\sum_{j=1}^{J(\epsilon_0)}|\tilde U_{j,n}|^2\tilde U_{j,n}\\  \\
H_{n,\epsilon_0}\Big|_{t=0}=\sum_{j=1}^{J(\epsilon_0)}\tilde U_{j,n}(0)
\end{array}
\right.
\end{equation*}
Hence, we can write
\begin{equation*}
H_{n,\epsilon_0}(t)=\ei\left(\sum_{j=1}^{J(\epsilon_0)}\tilde U_{j,n}(0)\right)+\int_0^t \mathrm{e}^{i(t-t')\Delta}
\sum_{j=1}^{J(\epsilon_0)}|\tilde U_{j,n}|^2\tilde U_{j,n} dt'.
\end{equation*}
Thus,
\begin{multline*}
D^{1/2}H_{n,\epsilon_0}(t)=\\=D^{1/2}\ei\left(\sum_{j=1}^{J(\epsilon_0)}\tilde U_{j,n}(0)\right)+D^{1/2}\int_0^t \mathrm{e}^{i(t-t')\Delta}
\sum_{j=1}^{J(\epsilon_0)}|\tilde U_{j,n}|^2\tilde U_{j,n} dt'\\=A(t)+B(t).
\end{multline*}
\begin{equation*}
||A(t)||_\hdt=\left\lVert\sum_{j=1}^{J(\epsilon_0)}\tilde U_{j,n}(0)\right\rVert_\hdt\leq1+
\left\lVert\sum_{j=1}^{J(\epsilon_0)}\frac{V_j^l\left(\frac{x-x_{j,n}}{\lambda_{j,n}},s_{j,n}\right)}{\lambda_{j,n}}\right\rVert_\hdt
\end{equation*}
for $n$ large depending on $J(\epsilon_0)$, where we have used the definition of the non-linear profile $U_j$. But, the second term on the right equals 
\begin{equation*}
||u_{0,n}-w_n^{J(\epsilon_0)}||_\hdt\leq6^{1/2}A_C,
\end{equation*}
by \eqref{eq37}.

On the other hand, Lemma \ref{lem21} and Lemma \ref{lem23} show that
\begin{equation*}
||B(t)||_\hdt\leq \sum_{j=1}^{J(\epsilon_0)}||U_j||^2_{S(s_{j,n},+\infty)}\cdot||D^{1/2}U_j||_{W(s_{j,n},+\infty)}\leq\tilde C_0
\end{equation*}
in light of \eqref{eq38}. The argument for $||D^{1/2}B||_{W(0,+\infty)}$ is similar, thus establishing \eqref{eq39}.

Next, let 
\begin{equation*}
  R_{n,\epsilon_0}=|H_{n,\epsilon_0}|^2 H_{n,\epsilon_0}-\sum_{j=1}^{J(\epsilon_0)}|\tilde U_{j,n}|^2\tilde U_{j,n}.
\end{equation*}
We claim that, for $n\geq n(\epsilon_0)$, we have 
\begin{equation}\label{eq310}
||D^{1/2}R_{n,\epsilon_0}||_{L^{5/3}_{(0,+\infty)}L^{30/23}_x}\xrightarrow[n\to\infty]{}0.
\end{equation}
A similar proof is given in \cite{KM} and \cite{KM2}, but we give the full details here to deal explicitly with the difficulties arising from the non-local character of $D^{1/2}$. 

To establish \eqref{eq310}, we need to study terms of the form 
\begin{equation*}
D^{1/2}( \tilde U_{j,n}\overline{\tilde U_{j',n}}\tilde U_{j_1,n}),
\end{equation*}
 in the $L^{5/3}_{(0,+\infty)}L^{30/11}_x$ norm, where at least two of $j,j',j_1$ are different. Assuming that $j\neq j'$, using the Leibniz rule for fractional integration (Theorem A.6 in \cite{KPV}) and H\"older's inequality, we are reduced to estimating the sum of 
\begin{equation*}
||\tilde U_{j_1,n}||_{S(0,+\infty)}||D^{1/2}(\tilde U_{j,n}\overline{\tilde U_{j',n}})||_{L^{5/2}_{(0,+\infty)}L^{30/17}_x} 
\end{equation*}
and 
\begin{equation*}
||\tilde U_{j,n}\overline{\tilde U_{j',n}}||_{L^{5/2}_{(0,+\infty)}L^{5/2}_x}||D^{1/2}\tilde U_{j_1,n}||_{W(0,+\infty)}.
\end{equation*}
The arguments in \cite{K} and \eqref{eq38} easily show that the second term goes to 0. \eqref{eq38} also gives that the first factor in the first term is bounded, thus reducing us to showing
\begin{equation}\label{eq311}
\lim_{n\to\infty}||D^{1/2}(\tilde U_{j,n}\overline{\tilde U_{j',n}})||_{L^{5/2}_{(0,+\infty)}L^{30/17}_x}=0.
\end{equation}
We proceed by considering cases.

Assume first that $\left(\frac{\lambda_{j,n}}{\lambda_{j',n}}\right)\to+\infty$. Apply the Leibniz rule in the $x$ variable, to bound the integral by
\begin{multline*}
\left\lVert\;||D^{1/2}\tilde U_{j,n}||_{L^{30/11}_x}||\tilde U_{j',n}||_{L^5_x}
\right\rVert_{L^{5/2}_{(0,+\infty)}}+\\+
\left\lVert\;||D^{1/2}\tilde U_{j',n}||_{L^{30/11}_x}||\tilde U_{j,n}||_{L^5_x}
\right\rVert_{L^{5/2}_{(0,+\infty)}}.
\end{multline*}
Change variables in the $x$ integrals. The terms then become 
\begin{multline*}
\frac{1}{\lambda_{j,n}^{2/5}}\frac{1}{\lambda_{j',n}^{2/5}}\left\lVert\;
\left\lVert(D^{1/2}U_{j})\left(\cdot,\frac{t-t_{j,n}}{\lambda_{j,n}^2}\right)\right\rVert_{L^{30/11}_x}\right.
\left.
\left\lVert U_{j'}\left(\cdot,\frac{t-t_{j',n}}{\lambda_{j',n}^2}\right)\right\rVert_{L^5_x}
\right\rVert_{L^{5/2}_{(0,+\infty)}}+\\+
\frac{1}{\lambda_{j,n}^{2/5}}\frac{1}{\lambda_{j',n}^{2/5}}\left\lVert\;
\left\lVert(D^{1/2}U_{j'})\left(\cdot,\frac{t-t_{j,n}}{\lambda_{j,n}^2}\right)\right\rVert_{L^{30/11}_x}\right.
\left.
\left\lVert U_{j}\left(\cdot,\frac{t-t_{j',n}}{\lambda_{j',n}^2}\right)\right\rVert_{L^5_x}
\right\rVert_{L^{5/2}_{(0,+\infty)}}.
\end{multline*}
To handle, say, the first term, we first make some observations, to be used throughout, about non-linear profiles. Note that if $T_+(U_j)=+\infty$ (as we are assuming), there exists $-\infty\leq a_j<+\infty$ so that
\begin{multline*}
\sup_{t\in(a_j,+\infty)}\!\!\!\!||U_j(t)||_\hdt+||U_j||_{S(a_j,+\infty)}+||D^{1/2}U_j||_{W(a_j,+\infty)}\leq\\
\leq C\left(
\sup_{t\in(s_{j,n},+\infty)}\!\!\!\!||U_j(t)||_\hdt+||U_j||_{S(s_{j,n},+\infty)}+||D^{1/2}U_j||_{W(s_{j,n},+\infty)}
\right),
\end{multline*}
for $n$ large, and $s_{j,n}\in(a_j,+\infty)$ for $n$ large. Let now 
\begin{equation*}
f_j(s)=||D^{1/2}U_j(\cdot,s)||_{L^{30/11}_x},\quad g_{j'}(s)=||U_{j'}(\cdot,s)||_{L^{5}_x},
\end{equation*}
 belonging to 
$L^5_{(a_j,+\infty)}$, $L^5_{(a_{j'},+\infty)}$ respectively. We can approximate $f_j$, $g_{j'}$ by 
$C_0^\infty(a_j,+\infty)$, $C_0^\infty(a_{j'},+\infty)$ functions respectively, making a small error in our term. We do the cange of variables $s=(t-t_{j',n})/{\lambda_{j',n}^2}$, to obtain (for $n$ large)
\begin{equation*}
\left(\frac{\lambda_{j',n}}{\lambda_{j,n}}\right)^{2/5}\left\lVert
f_j\left(s\frac{\lambda_{j',n}^2}{\lambda_{j,n}^2}+\frac{t_{j',n}-t_{j,n}}{\lambda_{j,n}^2}\right)\cdot 
g_{j'}(s)\right\rVert_{L^{5/2}_{(s_{j',n},+\infty)}}.
\end{equation*}
Note that for $s\in(s_{j',n},+\infty)$, 
\begin{equation*}
s\frac{\lambda_{j',n}^2}{\lambda_{j,n}^2}+\frac{t_{j',n}-t_{j,n}}{\lambda_{j,n}^2}\in(s_{j,n},+\infty)\subset(a_j,+\infty),
\end{equation*}
 so $f_j$ is bounded and since 
$(s_{j',n},+\infty)\subset(a_{j'},+\infty)$, the term tends to 0. The other term is analogous and the case 
$\left(\frac{\lambda_{j',n}}{\lambda_{j,n}}\right)\to+\infty$ is symmetric to this one. The next case (see (2.92) in \cite{K}) is when $\lambda_{j,n}=\lambda_{j',n}$ and ${|t_{j',n}-t_{j,n}|}/{\lambda_{j,n}^2}\to\infty$. By symmetry we can assume that 
$(t_{j',n}-t_{j,n})/{\lambda_{j,n}^2}\to+\infty$. Proceeding in exactly the same way, we see that the support asumption on $f_j$ makes the integral 0 for $n$ large. The final case is 
$\lambda_{j,n}=\lambda_{j',n}$, $|t_{j',n}-t_{j,n}|/{\lambda_{j,n}^2}\leq C$, $|x_{j',n}-x_{j,n}|/{\lambda_{j,n}}\to+\infty$. In this case we need to re-examine the proof of the Leibniz rule in \cite{KPV}, using Proposition A.2, Lemma A.3, and the proof of Theorem A.8 in \cite{KPV}. We then see that, for $1<p<\infty$, we have
\begin{equation}\label{eq312}
||D^{1/2}(f\cdot g)-fD^{1/2}(g)-gD^{1/2}(f)||_{L^p_x}\leq C||A(f)\cdot B(D^{1/2}g)||_{L^p_x},
\end{equation}
where $A$, $B$ are sublinear operators which commute with translations and which are $L^q_x$ bounded for any 
$q>1$. (The $A$ and $B$ are basically square functions plus maximal functions) Consider in our estimate for \eqref{eq311},
\begin{multline*}
\frac{1}{\lambda_{j,n}^2}\left\lVert D^{1/2}\left(U_j\left(\frac{x-x_{j,n}}{\lambda_{j,n}},
\frac{t-t_{j,n}}{\lambda_{j,n}^2}\right)\right.\right. \left.\left. 
U_{j'}\left(\frac{x-x_{j',n}}{\lambda_{j',n}},
\frac{t-t_{j',n}}{\lambda_{j',n}^2}\right)\right)\right\rVert_{L^{30/17}_x}=\\=
\frac{1}{\lambda_{j,n}^{4/5}}\left\lVert D^{1/2}\left(U_j\left(x-\frac{x_{j,n}}{\lambda_{j,n}},
\frac{t-t_{j,n}}{\lambda_{j,n}^2}\right)\right.\right. \left.\left. 
 U_{j'}\left(x-\frac{x_{j',n}}{\lambda_{j',n}},
\frac{t-t_{j',n}}{\lambda_{j',n}^2}\right)\right)\right\rVert_{L^{30/17}_x},
\end{multline*}
by change of variables. We then apply \eqref{eq312} and the triangle inequality. Consider, for instance, the term
\begin{multline*}
\frac{1}{\lambda_{j,n}^{4/5}}\left\lVert U_j\left(\cdot-\frac{x_{j,n}}{\lambda_{j,n}},
\frac{t-t_{j,n}}{\lambda_{j,n}^2}\right)\right. \left. 
D^{1/2}U_{j'}\left(\cdot-\frac{x_{j',n}}{\lambda_{j',n}},
\frac{t-t_{j',n}}{\lambda_{j',n}^2}\right)\right\rVert_{L^{30/17}_x}=
\\=
\frac{1}{\lambda_{j,n}^{4/5}}\left\lVert U_j\left(\cdot-\frac{x_{j,n}-x_{j',n}}{\lambda_{j,n}},
\frac{t-t_{j,n}}{\lambda_{j,n}^2}\right)\right. \left.
 D^{1/2}U_{j'}\left(\cdot,
\frac{t-t_{j',n}}{\lambda_{j',n}^2}\right)\right\rVert_{L^{30/17}_x}.
\end{multline*}
We need to take the $L^{5/2}_{(0,+\infty)}$ norm of this expression. But, by approximation in 
$L^5_{(a_j,+\infty)}L^5_x$ and $L^5_{(a_{j'},+\infty)}L^{30/11}_x$ by $C_0^\infty$ functions, we see that the 
$x$ integral will be 0 for large $n$. All the other terms are handled similarly, using that $A$, $B$ commute with translations. This finishes the proof of \eqref{eq311} and hence that of \eqref{eq310}.

Once \eqref{eq39} and \eqref{eq310} hold, we apply Theorem \ref{thm212}, with $\tilde u=H_{n,\epsilon_0}$,
$e=R_{n,\epsilon_0}$. Consider 
\begin{multline*}
v_{0,n}=\tilde u(0)-u_{0,n}=\\=\sum_{j=1}^{J(\epsilon_0)}\tilde U_{j,n}(0)-u_{0,n}=
\sum_{j=1}^{J(\epsilon_0)}\left[\tilde U_{j,n}(0)-\frac{1}{\lambda_{j,n}}V_j^l\left(\frac{x-x_{j,n}}{\lambda_{j,n}},
-\frac{t_{j,n}}{\lambda_{j,n}^2}\right)-\right]\\-w_n^{J(\epsilon_0)}.
\end{multline*}
By the properties of the non-linear profile, for $n$ large we have $||v_{0,n}||^2_\hdt\leq\epsilon_1+5A_C^2$, while 
\begin{equation*}
||\ei v_{0,n}||_{S(0,+\infty)}\leq C\epsilon_1+||\ei w_n^{J(\epsilon_0)}||_{S(0,+\infty)}\leq C\epsilon_1+
\epsilon_0,
\end{equation*}
for $n$ large. If $\epsilon_1$ and $\epsilon_0$ are chosen small, Theorem \ref{thm212} yields 
$||u_n||_{S(0,+\infty)}$ $<\infty$, a contradiction which establishes Step 2.

Because of Step 2, rearranging in $j$, we can find $1\leq J_1\leq J_0$ so that, for $1\leq j\leq J_1$ we have (for $n$ large) $||U_j||_{S(s_{j,n},T_+(U_j))}=+\infty$, and for $j>J_1$ we have $T_+(U_j)=+\infty$ and 
$||U_j||_{S(s_{j,n},\infty)}<+\infty$. As a consequence of Step 1 and Step 2, we now have
\begin{equation}\label{eq313}
\sum_{j\geq J_1}\!\!\sup_{t\in(s_{j,n},+\infty)}\!\!\!\!\!\!||U_j(t)||^2_\hdt\!+\!||D^{1/2}U_j||^2_{W(s_{j,n},\infty)}\!+\!
||U_j||^2_{S(s_{j,n},\infty)}\leq C_0,
\end{equation}
for $n$ large enough. Now, for $k\in\mathbb{N}$, $1\leq j\leq J_1$, define
\begin{equation*}
T_{j,k}^+=\left\{
\begin{array}{ll}
T_+(U_j)-\frac{1}{k}&\quad if\;T_+(U_j)<\infty\\ \\
k&\quad if \;T_+(U_j)=\infty
\end{array}\right.,
\end{equation*}
and $t_{j,k,}^n$ by $s_{j,k}+t_{j,k,}^n/\lambda_{j,n}^2=T^+_{j,k}$ and $t^n_k=\min_{1\leq j\leq J_1}t_{j,k,}^n$. With these definitions, $\tilde U_{j,n}$ is defined for $0\leq t\leq t^n_k$, for $j=1,\ldots$ and we have, for $n$ large,
\begin{equation}\label{eq314}
\sum_{j=1}^\infty\sup_{t\in(0,t^n_k)}\!\!\!||\tilde U_{j,n}(t)||^2_\hdt+||D^{1/2}\tilde U_{j,n}||^2_{W(0,t^n_k)}+
||\tilde U_{j,n}||^2_{S(0,t^n_k)}\leq C_k.
\end{equation}
Recall that, for $\epsilon>0$ given, we have:

\begin{eqnarray}\label{eq315}\begin{split}&\text{
There exists $J(\epsilon)$ such that, for $J\geq J(\epsilon)$, there exists}\\ 
&\text{$n(J,\epsilon)$ so that, for $n\geq n(J,\epsilon)$ we have}\\& 
||\ei w_n^J||_{S(-\infty,+\infty)}\leq\epsilon.
\end{split}\\ \nonumber\\\nonumber\\
\label{eq316}\begin{split}
&\text{For each fixed $J\geq1$, there exists $n(J,\epsilon)$ so that, for}\\&
\text{$n\geq n(J,\epsilon)$, we have }\\&
||u_{0,n}||^2_\hdt=\sum_{j=1}^J||V_{0.j}||^2_\hdt+||w_n^J||^2_\hdt+\epsilon(J,n),\\
&\text{with $|\epsilon(J,n)|\leq\epsilon$.}
\end{split}\\ \nonumber\\\nonumber\\
\label{eq317}\begin{split}
&\text{For each fixed $J\geq1$, there exists $n(J,\epsilon)$ so that, for }\\&
\text{$n\geq n(J,\epsilon)$, we have }\\& 
\sum_{j=1}^J\left\lVert\tilde U_{j,n}(x,0)-V_j^l\left(\frac{x-x_{j,n}}{\lambda_{j,n}},s_{j,n}\right)/
\lambda_{j,n}\right\rVert_\hdt\leq\epsilon
\end{split}
\end{eqnarray}
(This is a simple consequence of the definition of the non-linear profile.)

The next step will prove a crucial orthogonality.

\begin{step} For each fixed $J\geq1$, there exists $n(J,\epsilon)$  so that (after passing to a subsequence in $n$), for any $1\leq J_2\leq J$, $n\geq n(J,\epsilon)$, we have:
\begin{equation}\label{eq318}
\left|\left\lVert\sum_{j=J_2}^J\frac{V_j^l\left(\frac{x-x_{j,n}}{\lambda_{j,n}},s_{j,n}\right)}{\lambda_{j,n}}
\right\rVert^2_\hdt-
\sum_{j=J_2}^J\left\lVert\frac{V_j^l\left(\frac{x-x_{j,n}}{\lambda_{j,n}},s_{j,n}\right)}{\lambda_{j,n}}
\right\rVert^2_\hdt\right|\leq\epsilon.
\end{equation}
\end{step}

In order to establish \eqref{eq318}, we need to show that, after passing to a subsequence in $n$, for 
$J_2\leq j,j'\leq J$, $j\neq j'$, $J$ fixed, we have
\begin{equation}\label{eq319}
\lim_{n\to\infty}\frac{\left<(D^{1/2}V_j^l)\left(\frac{x-x_{j,n}}{\lambda_{j,n}},s_{j,n}\right),
(D^{1/2}V_{j'}^l)\left(\frac{x-x_{j',n}}{\lambda_{j',n}},s_{j',n}\right)
\right>}{\lambda_{j,n}^{3/2}\lambda_{j',n}^{3/2}}=0
\end{equation}
We will make repeated use of the following formula:
\begin{equation}\label{eq320}
\begin{split}
\left(\mathrm{e}^{it_0\Delta} v_0\right)\left(\frac{x-x_0}{\lambda_0}\right)&=
\left(\mathrm{e}^{i\lambda_0^2t_0\Delta}v_{0,\lambda_0,x_0}\right)(x),\\
\text{where } v_{0,\lambda_0,x_0}(x)&=v_0\left(\frac{x-x_0}{\lambda_0}\right).
\end{split}
\end{equation}
From \eqref{eq320}, it follows that
\begin{equation*}
V_j^l\left(\frac{x-x_{j,n}}{\lambda_{j,n}},s_{j,n}\right)=
\left(\mathrm{e}^{is_{j,n}\Delta} V_{0,j}\right)\left(\frac{x-x_{j,n}}{\lambda_{j,n}}\right)=
\mathrm{e}^{-it_{j,n}\Delta}V_{0,j,\lambda_{j,n},x_{j,n}}(x)
\end{equation*}
and similarly for $j'$. Thus, the right hand side in \eqref{eq319} becomes
\begin{equation}\label{eq321}
\frac{1}{\lambda_{j,n}^{3/2}\lambda_{j',n}^{3/2}}\left<
\mathrm{e}^{i(t_{j,n}-t_{j',n})\Delta}D^{1/2}V_{0,j',\lambda_{j',n},x_{j',n}}(x),
D^{1/2}V_{0,j}\left(\frac{x-x_{j,n}}{\lambda_{j,n}}\right)
\right>
\end{equation}
which we will show goes to 0 (after passing to a subsequence in $n$) because of orthogonality, $j\neq j'$. We consider various cases.

\subparagraph{Case 1:} $(\lambda_{j,n}/\lambda_{j',n})\to0$. Then we make the change of variables 
$y=(x-x_{j,n})/\lambda_{j,n}$, and \eqref{eq321} becomes
\begin{equation*}
\left(\frac{\lambda_{j,n}}{\lambda_{j',n}}\right)^{3/2}\left<
\mathrm{e}^{i(t_{j,n}-t_{j',n})/\lambda_{n,j}^2\Delta}D^{1/2}V_{0,j',\frac{\lambda_{j',n}}{\lambda_{j,n}},
\frac{x_{j,n}-x_{j',n}}{\lambda_{j,n}}}(x),
D^{1/2}V_{0,j}(x)
\right>.
\end{equation*}
We now consider
\subparagraph{Case 1a):} $|t_{j,n}-t_{j',n}|/\lambda_{j,n}^2\leq C$. In this case, after passing to a subsequence, $(t_{j,n}-t_{j',n})/\lambda_{j,n}^2\to s_0$. Then,
\begin{equation*}
\mathrm{e}^{-i(t_{j,n}-t_{j',n})/\lambda_{j,n}^2\Delta}D^{1/2}V_{0,j}\to\mathrm{e}^{-is_0\Delta}D^{1/2}V_{0,j} 
\end{equation*}
in $L^2$ and we are reduced to considering
\begin{equation*}
\left(\frac{\lambda_{j,n}}{\lambda_{j',n}}\right)^{3/2}\left<
D^{1/2}V_{0,j',\frac{\lambda_{j',n}}{\lambda_{j,n}},
\frac{x_{j,n}-x_{j',n}}{\lambda_{j,n}}}(x),
\mathrm{e}^{-is_0\Delta}D^{1/2}V_{0,j}(x)
\right>.
\end{equation*}
We now approximate $D^{1/2}V_{0,j'}$ and $\mathrm{e}^{-is_0\Delta}D^{1/2}V_{0,j}(x)$ by $C_0^\infty$ functions in the $L^2$ norm and we readily see that this goes to 0.

\subparagraph{Case 1b):} $|t_{j,n}-t_{j',n}|/\lambda_{j,n}^2$ is not bounded. Then, after passing to a subsequence, $(t_{j,n}-t_{j',n})/\lambda_{j,n}^2\to+\infty$ (say). Let 
$\overline s_n=(t_{j,n}-t_{j',n})/\lambda_{j,n}^2$, and let 
\begin{equation*}
h_n(x)=\left(\frac{\lambda_{j,n}}{\lambda_{j',n}}\right)\mathrm{e}^{i\overline s_n\Delta}V_{0,j',
\frac{\lambda_{j',n}}{\lambda_{j,n}},\frac{x_{j,n}-x_{j',n}}{\lambda_{j,n}}}(x).
\end{equation*}
By \eqref{eq37}, for $n$ large, $||h_n||_\hdt\leq5A_C$. Moreover, using \eqref{eq320},
\begin{multline*}
\ei h_n(x)=\left(\frac{\lambda_{j,n}}{\lambda_{j',n}}\right)\\\times
\mathrm{e}^{i\left[
({\lambda_{j',n}}/{\lambda_{j,n}})^2(t+\overline s_n)\right]\Delta}\;V_{0,j'}
\left(\left(\frac{\lambda_{j,n}}{\lambda_{j',n}}\right)\left(x-\left(\frac{x_{j,n}-x_{j',n}}{\lambda_{j,n}}
\right)\right)\right).
\end{multline*}
Note that $\overline s_n({\lambda_{j',n}}/{\lambda_{j,n}})^2\to+\infty$. Hence a change of variables shows that
\begin{equation*}
\left\lVert\ei h_n\right\rVert_{S(0,+\infty)}\to0,
\end{equation*}
so that Lemma \ref{lem36} gives the desired result. 

If $(t_{j,n}-t_{j',n})/\lambda_{j,n}^2\to-\infty$, we use $S(-\infty,0)$ and the corresponding version of Lemma \ref{lem36}. The case $(\lambda_{j',n}/\lambda_{j,n})\to0$ is symmetric. Thus we can now assume 
$\lambda_{j,n}=\lambda_{j',n}$ (see (2.92) in \cite{K}).

\subparagraph{Case 2:} $\lambda_{j,n}=\lambda_{j',n}$, $|t_{j,n}-t_{j',n}|/\lambda_{j,n}^2\to\infty$. This case is handled using the proof of Case 1b).

\subparagraph{Case 3:} $\lambda_{j,n}=\lambda_{j',n}$, $|t_{j,n}-t_{j',n}|/\lambda_{j,n}^2\leq C$, and 
$|(x_{j,n}-x_{j',n})/\lambda_{j,n}|$ $\to\infty$. In this case let $s_0$ be as in Case 1a). As in that case, we are reduced to studying 
\begin{equation*}
\left<
D^{1/2}V_{0,j'}\left(x-\frac{x_{j,n}-x_{j',n}}{\lambda_{j,n}}\right),D^{1/2}\mathrm{e}^{is_0\Delta}V_{0,j}(x)
\right>,
\end{equation*}
which goes to zero by aproximating $D^{1/2}V_{0,j'}$ and $D^{1/2}\mathrm{e}^{is_0\Delta}V_{0,j}$ in $L^2$ by $C_0^\infty$ functions. Thus Step 3 is established. 

The following step will be needed to apply Theorem \ref{thm212}.
\begin{step} 
For $J$, $n$ given, consider
\begin{equation*}
e^{(1)}_{J,n}(x,t)=f(\tilde U_{J,n})-\sum_{j=1}^Jf(\tilde U_{j,n}),
\end{equation*}
where $\tilde U_{J,n}(x,t)=\sum_{j=1}^J \tilde U_{j,n}(x,t)$ and where we recall that $f(z)=|z|^2z$, and
\begin{equation*}
e^{(2)}_{J,n}(x,t)=f(\tilde U_{J,n}+w_n^{l,J})-f(\tilde U_{J,n}),
\end{equation*}
where $w_n^{l,J}(x,t)=\ei w_n^J(x)$. Then,
\begin{itemize}
\item[i)] For each fixed $J\geq1$, $k\in\mathbb{N}$, there exists $n(J,k,\epsilon)$ so that, for
$n\geq n(J,k,\epsilon)$, we have
\begin{equation}\label{eq322}
||D^{1/2}e^{(1)}_{J,n}||_{L^{5/3}(0,t_k^n)L_x^{30/23}}\leq\epsilon.
\end{equation}
\item[ii)] For each fixed $k\in\mathbb{N}$, there exists $J=J(k,\epsilon)$ so that, for $J\geq J(k,\epsilon)$, there exists $n(J,k,\epsilon)$ so that, for
$n\geq n(J,k,\epsilon)$, we have
\begin{equation}\label{eq323}
||D^{1/2}e^{(2)}_{J,n}||_{L^{5/3}(0,t_k^n)L_x^{30/23}}\leq\epsilon.
\end{equation}
\end{itemize}
\end{step}
\begin{proof}
First, note that for $1\leq j\leq J_1$, we must have $s_j<+\infty$, otherwise we would have 
$||U_j||_{S(s_{j,n},+\infty)}<\infty$, by the construction of the nonlinear profile. Thus, for $k$ fixed, 
$1\leq j\leq J_1$, ther exists $-\infty\leq a_j<+\infty$, such that 
\begin{equation*}
\sup_{t\in(a_j,T^+_{j,k})}\!\!\!||D^{1/2}U_j(t)||_{L_x^2}+||U_j||_{S(a_j,T^+_{j,k})}+
||D^{1/2}U_j||_{W(a_j,T^+_{j,k})}<\infty
\end{equation*}
and $(s_{j,n},T^+_{j,k})\subset(a_j,T^+_{j,k})$, for $n$ large. Moreover, for $j>J_1$, $T^+(U_j)=+\infty$ and there exists $a_j$ with $-\infty\leq a_j<+\infty$ so that
\begin{multline*}
\sup_{t\in(a_j,+\infty)}\!\!\!\!||D^{1/2}U_j(t)||_{L_x^2}+||U_j||_{S(a_j,+\infty)}+
||D^{1/2}U_j||_{W(a_j,+\infty)}\leq\\\leq C\left[
\sup_{t\in(s_{j,n},+\infty)}\!\!\!\!||D^{1/2}U_j(t)||_{L_x^2}+||U_j||_{S(s_{j,n},+\infty)}+\right.\\\left.+
||D^{1/2}U_j||_{W(s_{j,n},+\infty)}
 \right].
\end{multline*}
Once these remarks are made, the proof of \eqref{eq322} is the same as the one of \eqref{eq310}, using 
\eqref{eq314}. The argument for ii) follows closely that of Keraani in the proof of Proposition  3.4 \cite{K}: 

First note that, in light of \eqref{eq314} and the remarks above, given $\epsilon_1>0$, there exists 
$J(\epsilon_1)\geq 1$ so that 
\begin{multline}\label{eq324}
\sum_{j\geq J(\epsilon_1)}\sup_{t\in(a_j,+\infty)}\!\!\!\!||U_j(t)||^2_\hdt+\\+||U_j||^2_{S(a_j,+\infty)}+
||D^{1/2}U_j||^2_{W(a_j,+\infty)}\leq\epsilon_1
\end{multline}
Also, from \eqref{eq37}, for any $J\geq1$, there exists $n(J)$ so that, for $n\geq n(J)$, 
\begin{equation}\label{eq325}
||w_n^J||^2_\hdt\leq 5A_C^2
\end{equation}
Finally, note that there exists $\tilde C_k$ so that, given $J\geq1$, there exists $n(J,k)$ so that, for 
$n\geq n(J,k)$ we have
\begin{equation}\label{eq326}
||\tilde U_{J,n}||_{S(0,t_k^n)}+
||D^{1/2}\tilde U_{J,n}||_{W(0,t_k^n)}+\sup_{t\in(0,t_k^n)}||\tilde U_{J,n}(t)||_\hdt\leq\tilde C_k
\end{equation}
The proof of \eqref{eq326} is simillar to the one of \eqref{eq38}, using \eqref{eq314}.

Next we write $f(\tilde U_{J,n}+w_n^{l,J})-f(\tilde U_{J,n})$ by expanding the cubic, term by term. In the analysis that follows, the worse kind of term is 
\begin{equation*}
D^{1/2}(|\tilde U_{J,n}|^2w_n^{l,J}), 
\end{equation*}
which we handle now. We estimate this by using Theorem A.8 in \cite{KPV} in the form 
\begin{equation*}
||D^{1/2}(f\cdot g)-fD^{1/2}g-gD^{1/2}f||_{L_t^pL_x^r}\leq C||D^{1/4}f||_{L_t^{p_1}L_x^{r_1}}
||D^{1/4}g||_{L_t^{p_2}L_x^{r_2}},
\end{equation*}
\begin{equation*}
\text{where }\quad \frac{1}{p}=\frac{1}{p_1}+\frac{1}{p_2},\quad\frac{1}{r}=\frac{1}{r_1}+\frac{1}{r_2}
\end{equation*}
Using \eqref{eq315},\eqref{eq325},\eqref{eq326} and interpolation, we are reduced to handling the worse term,
\begin{equation*} 
|||\tilde U_{J,n}|^2D^{1/2}w_n^{l,J}||_{L^{5/3}_{(0,t_k^n)}L_x^{30/23}}.
\end{equation*}
 Using \eqref{eq326} again and 
H\"older, we are reduced to showing that 
\begin{equation*}
||\tilde U_{J,n}D^{1/2}w_n^{l,J}||_{L^{5/2}_{(0,t_k^n)}L_x^{30/17}}
\end{equation*}
 is small for $J$ large, $n$ large. Using the argument in \eqref{eq39}, together with \eqref{eq324}, \eqref{eq37}, \eqref{eq318} and the definition of the non-linear profile, we see that the norms of $\sum_{J(\epsilon_1)}^J\tilde U_{j,n}$ are smaller that 
$10\epsilon_1$, uniformly in $J$, for $n$ large depending on $J$. W are thus reduced to showing that for each fixed $j$, we have that 
\begin{equation*}
||\tilde U_{j,n}D^{1/2}w_n^{l,J}||_{L^{5/2}_{(0,t_k^n)}L_x^{30/17}}
\end{equation*}
 is small, for large $J$ and $n$. Let us consider first $1\leq j\leq J_1$, Then $t_k^n\leq t_{j,k}^n$. Change variables 
$y=(x-x_{j,n})/\lambda_{j,n}$, $s=s_{j,n}+t/\lambda_{j,n}^2$ and define 
\begin{equation*}
D^{1/2}\tilde w_{n,j}^J(y,s)=\lambda_{j,n}^{3/2}D^{1/2}w_n^{l,J}(\lambda_{j,n}y+x_{j,n},\lambda_{j,n}^2s-
\lambda_{j,n}^2s_{j,n}).
\end{equation*}
The integral we are considering is bounded by 
\begin{equation*}
||U_jD^{1/2}\tilde w_{n,j}^J||_{L^{5/2}_{(a_j,T^+_{j,k})}L^{30/17}_y}.
\end{equation*}
 Note that 
\begin{equation*}
||D^{1/2}\tilde w_{n,j}^J||_{L^{5}_{s}L^{30/11}_y}=||D^{1/2} w_{n,j}^J||_{L^{5}_{t}L^{30/11}_x}
\end{equation*}
and
\begin{equation*}
||\tilde w_{n,j}^J||_{L^{5}_{s}L^{5}_y}=||w_{n,j}^J||_{L^{5}_{t}L^{5}_x}.
\end{equation*}
Since $U_j\in L^{5}_{(a_j,T^+_{j,k})}L^{30/17}_y$, by H\"older's inequality and density, we can assume 
$U_j\in C_0^\infty(B)$, $B$ a bounded subset of $\R^4$. It thus suffices to show that
$||D^{1/2}\tilde w_{n,j}^J||_{L^2(B)}$ can be made small, by first choosing $J$ large and then $n$ large. Note from \eqref{eq320} that $\tilde w_{n,j}^J(y,s)=(\mathrm{e}^{is\Delta}\tilde w_{0,n,j})(y)$, where
$\tilde w_{0,n,j}=\lambda_{j,n}w_n^{l}(\lambda_{j,n}y+x_{j,n},-\lambda_{j,n}^2s_{j,n})$. The desired result follows from:
\begin{lem}\label{lem327}
Let $B$ be a bounded subset of $\R^3\times\R$. Then, for any $\eta>0$, there exists $C_\eta>0$ such that
\begin{equation*}
||D^{1/2}v||_{L^2(B)}\leq C_\eta||v||_{L^5(\R^4)}+\eta||v(0)||_\hdt,
\end{equation*}
where $v$ is a solution to the linear Schr\"odinger equation.
\end{lem}
The proof of Lemma \ref{lem327} is analogous to the one of Lemma 3.7 in \cite{K}.

Finally, the case $j>J_1$ follows similarly, replacing $(0,t_k^n)$ by $(a_j,+\infty)$. This concludes the proof of Step 4.
 $\square$ \end{proof}

Fix now $k\in\mathbb{N}$. Choose $J(m,k)$ so that, for $n\geq n_1(J,m,k)$, $J\geq J(m,k)$ we have (by 
\eqref{eq315} and \eqref{eq323}),
\begin{equation*}
||\ei w_n^J||_{S(-\infty,+\infty)}\leq\frac{1}{m},\quad
||D^{1/2}e_{J,n}^{(2)}||_{L^{5/3}_{(0,t_k^n)}L_x^{30/23}}\leq\frac{1}{2m}.
\end{equation*}
Next, choose for $J=J(m,k)$ fixed, $n(m,k)\geq n_1(J(m),m,k)$, so large that $|\epsilon(J,n(m,k))|\leq 1/m$,
($\epsilon(J,n)$ as in \eqref{eq316}) so that \eqref{eq317} holds with $\epsilon=1/m$, $n=n(m,k)$, so that
\eqref{eq318} holds with $\epsilon=1/(2m)^2$, $n=n(m,k)$, \eqref{eq322} holds with $J=J(m,k)$, $n=n(m,k)$,
$\epsilon=1/(2m)$. We can also ensure, in our choices, that $J(m+1,k)>J(m,k)$, $n(m,k)<n(m+1,k)$.

\begin{step}
For $0\leq t\leq t_k^{n(m,k)}$ and $m$ large. we have $t_k^{n(m,k)}\leq T_+(u_{0,n(m,k)})$ and 
\begin{equation*}
u_{n(m,k)}(t)=\tilde U_{J(m,k),n(m,k)}(t)+w_{n(m,k)}^{l,J(m,k)}(t)+r_{m,k}(t),
\end{equation*}
where
\begin{multline*}
\sup_{0\leq t\leq t_k^{n(m,k)}}\!\!\!\!\!||r_{m,k}(t)||_\hdt+\\+
||r_{m,k}||_{S(0,t_k^{n(m,k)})}+
||D^{1/2}r_{m,k}||_{W(0,t_k^{n(m,k)})}=\epsilon_k(m)
\end{multline*}
with $\epsilon_k(m)\xrightarrow[m\to\infty]{}0$.
\end{step}
\begin{proof}
Define 
\begin{equation*}
\tilde u(x,t)=\tilde U_{J(m,k),n(m,k)}(x,t)+w_{n(m,k)}^{l,J(m,k)}(x,t).
\end{equation*}
Let 
\begin{equation*}
e(x,t)=f(\tilde U_{J(m,k),n(m,k)}+w_{n(m,k)}^{l,J(m,k)})-\sum_{j=1}^{J(m,k)}f(\tilde U_{J(m,k),n(m,k)}).
\end{equation*}
Note that, form Step 4 and our choice of $J(m,k)$, $n(m,k)$, we have 
\begin{equation*}
||D^{1/2}e||_{L^{5/3}_{(0,t_k^{n(m,k)})}L_x^{30/23}}\leq\frac{1}{m}. 
\end{equation*}
Notice also that
$i\partial_t\tilde u+\Delta\tilde u-|\tilde u|^2\tilde u=-e$. Also, from our choices of $J(m,k)$, $n(m,k)$, we have $||u_{n(m,k)}(x,0)-\tilde u(x,0)||_\hdt\leq 1/m$. Then Step 5 follows from \eqref{eq325}, \eqref{eq326} and Theorem \ref{thm212}.
 $\square$ \end{proof}

\begin{step}
There exists $j_0$, $1\leq j_0\leq J_1$, a subsequence $\{k_\alpha\}$, $k_\alpha\xrightarrow[\alpha\to\infty]{}\infty$, and, for each fixed $k_\alpha$, a subsequence $n(m_\beta(k_\alpha),k_\alpha)\uparrow+\infty$ as $\beta\to+\infty$, for each $k_\alpha$, with 
$m_\beta(k_\alpha)\uparrow+\infty$ as $\beta\to+\infty$, so that
\begin{equation*}
t_{j_0,k_\alpha}^{n(m_\beta(k_\alpha),k_\alpha)}=t_{k_\alpha}^{n(m_\beta(k_\alpha),k_\alpha)}
\end{equation*}
for each $\alpha$, $\beta$.
\end{step}
\begin{proof}
Notice that for each fixed $k$, there exists $j(k)$ so that $1\leq j(k)\leq J_1$ and 
\begin{equation*}
t_{k}^{n(m,k)}=t_{j(k),k}^{n(m,k)}
\end{equation*}
for infinitely many $m$'s. Furthermore, there exists $j_0$, $1\leq j_0\leq J_1$ so that $j(k)=j_0$ for infinitely may $k$'s.
 $\square$ \end{proof}

Recall that $||U_{j_0}||_{S(s_{j_0,n(m,k)},T^+(U_{j_0})}=+\infty$, for all large $m$, for fixed $k$, and that 
$s_{j_0}=\lim_{n\to\infty}s_{j_0,n}<+\infty$. We can then find $-\infty<b_{j_0}<T^+(U_{j_0})$ so that 
$s_{j_0,n(m,k)}\leq b_{j_0}$ for all large m and fixed $k$, so that $||U_{j_0}||_{S(b_{j_0},T^+(U_{j_0})}=\infty$. By definition of $A_C$, we have
\begin{equation}\label{eq328}
A^2=\sup_{t\in(b_{j_0},T^+(U_{j_0})}||U_{j_0}(t)||^2_\hdt\geq A_C^2.
\end{equation}
Also, $A_k^2=\sup_{t\in(b_{j_0},T^+_{j_0,k})}||U_{j_0}(t)||^2_\hdt$ verifies $\lim_{k\to\infty}A_k^2=A^2$.

Now, let $T_{j_0,k}\in[b_{j_0},T^+_{j_0,k}]$ be such that $A_k^2=||U_{j_0}(T_{j_0,k})||^2_\hdt$. Define $\tau_{j_0,k}^{n(m,k)}$ by the formula
\begin{equation*}
s_{j_0,n(m,k)}+\frac{\tau_{j_0,k}^{n(m,k)}}{\lambda_{j,n(m,k)}^2}=T_{j_0,k}.
\end{equation*}
Note that for fixed $k$, $m$ large, $\tau_{j_0,k}^{n(m,k)}\geq 0$. Also, since $T_{j_0,k}\leq T_{j,k}^+$, 
$\tau_{j_0,k}^{n(m,k)}\leq t_{j_0,k}^{n(m,k)}$. Since 
$t_{j_0,k_\alpha}^{n(m_\beta(k_\alpha),k_\alpha)}=t_{k_\alpha}^{n(m_\beta(k_\alpha),k_\alpha)}$, for all $\alpha$, $\beta$, we have that\\
$\tilde U_{j,n(m_\beta(k_\alpha),k_\alpha)}(\tau^{n(m_\beta(k_\alpha),k_\alpha)}_{j_0,k_\alpha})$ is defined for all $j$. 

The last step that we need is

\begin{step}
For each $k_\alpha$ fixed and $\beta$ large (after possibly taking a subsequence in $\beta$, which may depend on $k_\alpha$), we have:
\begin{multline}\label{eq329}
\left\lVert u_{n(m_\beta(k_\alpha)}(\tau_{j_0,k_\alpha}^{n(m_\beta(k_\alpha),k_\alpha)})\right\rVert^2_\hdt=\\=
\sum_{j=1}^{J(m_\beta(k_\alpha),k_\alpha)} \left\lVert\tilde U_{j,n(m_\beta(k_\alpha),k_\alpha)} (\tau_{j_0,k_\alpha}^{n(m_\beta(k_\alpha),k_\alpha)})\right\rVert^2_\hdt+\\+
\left\lVert w_{n(m_\beta(k_\alpha),k_\alpha)}^{l,J(m_\beta(k_\alpha),k_\alpha)}(\tau_{j_0,k_\alpha}^{n(m_\beta(k_\alpha),k_\alpha)})\right\rVert^2_\hdt+\epsilon_{k_\alpha}(\beta),
\end{multline}
where $\epsilon_{k_\alpha}(\beta)\xrightarrow[\beta\to\infty]{}0$.
\end{step}
\begin{proof}
In order to alleviate notation, in this proof we will simply write 
$k_\alpha=k$, $J(m_\beta(k_\alpha),k_\alpha)=J$, $n(m_\beta(k_\alpha),k_\alpha)=n$, 
$\tau_{j_1,k_\alpha}^{n(m_\beta(k_\alpha),k_\alpha)}=\tau_{j_1,k}^{n}$ and recall that $k$ is fixed and 
$J$, $n$ are large.

The first claim is that, given $\epsilon>0$, we can find $J_2=J_2(\epsilon)$, and $\beta_0(\epsilon)$ large, so that, for $\beta\geq\beta_0$, we have
\begin{equation}\label{eq330}
\sup_{0\leq t\leq t_k^n}\left\lVert\sum_{j=J_2}^J\tilde U_{j,n}(t)\right\rVert_\hdt\leq\epsilon.
\end{equation}
To establish \eqref{eq330}, note that, from \eqref{eq37} we have, for $\epsilon_1$ to be chosen, 
\begin{equation*}
\sum_{j=J_2(\epsilon_1)}^J||V_{0,j}||^2_\hdt\leq\epsilon_1^2
\end{equation*}
 and from Step 1,
\begin{multline*}
\sum_{j=J_2(\epsilon_1)}^\infty\sup_{t\in(-\infty,+\infty)}\!\!\!\!\!||U_j(t)||^2_\hdt+\\+
||U_j||^2_{S(-\infty,+\infty))}+
||D^{1/2}U_j||^2_{W(-\infty,+\infty)}\leq C\epsilon_1^2.
\end{multline*} 
Next, we use the integral equation for $\tilde U_{j,n}$, to see that, for \\ $0\leq t\leq+\infty$ we have
\begin{equation*}
\sum_{j=J_2}^J\tilde U_{j,n}(t)=\ei\left(\sum_{j=J_2}^J\tilde U_{j,n}(0)\right)+
\sum_{j=J_2}^J\int_0^t\mathrm{e}^{i(t-t')\Delta}f(\tilde U_{j,n})(t')dt'.
\end{equation*}
By Lemma \ref{lem21}, we have:
\begin{multline*}
\left\lVert\sum_{j=J_2}^J\tilde U_{j,n}(t)\right\rVert_\hdt\leq
\left\lVert\sum_{j=J_2}^J\tilde U_{j,n}(0)\right\rVert_\hdt+\\+
C\sum_{j=J_2}^J||\tilde U_{j,n}||^2_{S(-\infty,+\infty))}||D^{1/2}\tilde U_{j,n}||_{W(-\infty,+\infty))}
\leq\\\leq
\left\lVert\sum_{j=J_2}^J V_j^l\left(\frac{x-x_{j,n}}{\lambda_{j,n}},s_{j,n}\right)/\lambda_{j,n}
\right\rVert_\hdt+\frac{1}{m_\beta}+C\epsilon_1,
\end{multline*}
where we have used \eqref{eq317}, Step 1, Cauchy--Schwartz, and our choice of $J_2$.
Next, from \eqref{eq318}, we have that 
\begin{multline*}
\left\lVert\sum_{j=J_2}^J V_j^l\left(\frac{x-x_{j,n}}{\lambda_{j,n}},s_{j,n}\right)/\lambda_{j,n}
\right\rVert_\hdt\leq\\\leq
\left(\sum_{j=J_2}^J||V_{0,j}^2||^2_\hdt\right)^{1/2}+\frac{1}{m_\beta}\leq
\epsilon_1+\frac{1}{m_\beta},
\end{multline*}
and the first claim follows.

Next, note that in light of \eqref{eq330}, Step 5, \eqref{eq325} and \eqref{eq326}, in order to establish 
\eqref{eq329} it suffices to show :
\begin{equation}\label{eq331}
\begin{split}
&\text{for $1\leq j,j'\leq J_2$, $J_2$ fixed, $j\neq j'$,then }\\
&\left<D^{1/2}\tilde U_{j,n}(\tau_{j_0,k}^{n}),D^{1/2}\tilde U_{j',n}(\tau_{j_0,k}^{n})\right>\\
&\text{ tends to 0 with $\beta$ (after passing to a subsequence).}
\end{split}
\end{equation}
\begin{equation}\label{eq332}
\begin{split}
&\text{for $1\leq j\leq J_2$, $J_2$ fixed, }\\
&\left<D^{1/2}\tilde U_{j,n}(\tau_{j_0,k}^{n}),D^{1/2}w_n^{l,J}(\tau_{j_0,k}^{n})\right>\\
&\text{ tends to 0 with $\beta$ (after passing to a subsequence).}
\end{split}
\end{equation}

We now prove \eqref{eq331}. Let us define 
\begin{equation}\label{eq333}
\begin{split}
\tilde t_{j,n}&=-\frac{t_{j,n}}{\lambda_{j,n}^2}+\frac{\tau_{j_0,k}^n}{\lambda_{j,n}^2}\\
\tilde t_{j',n}&=-\frac{t_{j',n}}{\lambda_{j',n}^2}+\frac{\tau_{j_0,k}^n}{\lambda_{j',n}^2}
\end{split}
\end{equation}
Assume first that (say) $|\tilde t_{j',n}|\leq C_{j'}$. Then, after passing to a subsequence in $\beta$, we can assume that $\tilde t_{j',n}\to\tilde t_{j'}$. Note that
\begin{equation*}
-\frac{t_{j',n}}{\lambda_{j',n}^2}+\frac{\tau_{j_0,k}^n}{\lambda_{j',n}^2}\leq
-\frac{t_{j',n}}{\lambda_{j',n}^2}+\frac{t_{k}^n}{\lambda_{j',n}^2}
\end{equation*}
and $\tau_{j_0,k}^n\geq 0$ so that $U_{j'}(t)$ is continuous in $\hdt$ in a neighborhood of $t_{j'}$. Because of this and \eqref{eq324}, we only need to consider
\begin{equation*}
\frac{1}{\lambda_{j,n}^{3/2}\lambda_{j',n}^{3/2}}
\left<D^{1/2}U_j\left(\frac{x-x_{j,n}}{\lambda_{j,n}},\tilde t_{j,n}\right),
D^{1/2}U_{j'}\left(\frac{x-x_{j',n}}{\lambda_{j',n}},\tilde t_{j'}\right)
\right>.
\end{equation*}
We proceed by analyzing cases.

Assume that $(\lambda_{j,n}/\lambda_{j',n})\to+\infty$. If $|\tilde t_{j,n}|\leq C_{j}$, after passing to a subsequence we can assume $\tilde t_{j,n}\to\tilde t_{j}$ and we need only to consider
\begin{equation*}
\frac{1}{\lambda_{j,n}^{3/2}\lambda_{j',n}^{3/2}}
\left<D^{1/2}U_j\left(\frac{x-x_{j,n}}{\lambda_{j,n}},\tilde t_{j}\right),
D^{1/2}U_{j'}\left(\frac{x-x_{j',n}}{\lambda_{j',n}},\tilde t_{j'}\right)
\right>.
\end{equation*}
By approximating $D^{1/2}U_j(\tilde t_{j})$, $D^{1/2}U_{j'}(\tilde t_{j'})$ by $C_0^\infty$ functions in $L_x^2$, this case follows. If $|\tilde t_{j,n}|$ is not bounded, after passing to a subsequence,
$\tilde t_{j,n}\to\pm\infty$. Since for $j\leq J_1$, $\tilde t_{j,n}\leq T^+_{j,k}<\infty$, we must have, if 
$\tilde t_{j,n}\to+\infty$, that $j>J_1$ and $U_j$ scatters at $+\infty$.
If $\tilde t_{j,n}\to-\infty$, then, since $\tilde t_{j,n}\geq s_{j,n}$, $s_j=\lim_n s_{j,n}=-\infty$. Then, by construction of the non-linear profile, $U_j$ scatters at $-\infty$. In either case, there exists $h_j\in\hdt$ so that
\begin{equation*}
||U_j(\tilde t_{j,n})-\mathrm{e}^{i\tilde t_{j,n}\Delta}h_j||_\hdt\xrightarrow[\beta]{}0.
\end{equation*}
We can then replace 
\begin{equation*}
\frac{1}{\lambda_{j,n}^{3/2}}D^{1/2}U_j\left(\frac{x-x_{j,n}}{\lambda_{j,n}},\tilde t_{j,n}\right)
\end{equation*}
 by 
\begin{equation*}
\frac{1}{\lambda_{j,n}^{3/2}}D^{1/2}\mathrm{e}^{i\tilde t_{j,n}\Delta}h_j\left(\frac{x-x_{j,n}}{\lambda_{j,n}}\right),
\end{equation*}
 and consider
\begin{equation*}
\frac{1}{\lambda_{j,n}^{3/2}\lambda_{j',n}^{3/2}}
\left<(D^{1/2}\mathrm{e}^{i\tilde t_{j,n}\Delta}h_j)\left(\frac{x-x_{j,n}}{\lambda_{j,n}}\right),
(D^{1/2}U_{j'})\left(\frac{x-x_{j',n}}{\lambda_{j',n}},\tilde t_{j'}\right)
\right>.
\end{equation*}
We now use \eqref{eq320} and Lemma \ref{lem36} for $t>0$ or $t<0$, according to the limit of $\tilde t_{j,n}$, to conclude. The case  $(\lambda_{j,n}/\lambda_{j',n})\to0$ is completely analogous.

The next case is $\lambda_{j,n}=\lambda_{j',n}$, $|t_{j,n}-t_{j',n}|/\lambda_{j,n}^2\to+\infty$. In this case, since $\tilde t_{j,n}-\tilde t_{j',n}=(-t_{j,n}+t_{j',n})/\lambda_{j,n}^2$ and $|\tilde t_{j',n}|\leq C_{j'}$, we see that $|\tilde t_{j,n}|$ is unbounded. But then, the argument above applies, giving the proof in this case.

The final case is $\lambda_{j,n}=\lambda_{j',n}$, $|t_{j,n}-t_{j',n}|/\lambda_{j,n}^2\leq C$ and 
$|(x_{j,n}-x_{j',n})/\lambda_{j,n}|\to+\infty$. In this case , $|\tilde t_{j,n}|\leq C_{j}$ and we are reduced to considering
\begin{equation*}
\frac{1}{\lambda_{j,n}^{3}}
\left<D^{1/2}U_j\left(\frac{x-x_{j,n}}{\lambda_{j,n}},\tilde t_{j}\right),
D^{1/2}U_{j'}\left(\frac{x-x_{j',n}}{\lambda_{j,n}},\tilde t_{j'}\right)
\right>.
\end{equation*}
A change of variables and approximation bt $C_0^\infty$ functions yields this case.

By symmetry, we are reduced then to consider the case when both $\tilde t_{j,n}$ and $\tilde t_{j',n}$ are unbounded. Asume (say) $\tilde t_{j,n}\to+\infty$, $\tilde t_{j',n}\to+\infty$. By scattering, we are reduced to considering
\begin{equation*}
\frac{1}{\lambda_{j,n}^{3/2}\lambda_{j',n}^{3/2}}
\left<D^{1/2}\mathrm{e}^{i\tilde t_{j,n}\Delta}h_{j}\left(\frac{x-x_{j,n}}{\lambda_{j,n}}\right),
D^{1/2}\mathrm{e}^{i\tilde t_{j',n}\Delta}h_{j'}\left(\frac{x-x_{j',n}}{\lambda_{j',n}}\right)
\right>.
\end{equation*}
But, using \eqref{eq333} and \eqref{eq320}, we see that this equals 
\begin{equation*}
\frac{1}{\lambda_{j,n}^{3/2}\lambda_{j',n}^{3/2}}
\left<(D^{1/2}h_{j})_{\lambda_{j,n},x_{j,n}}(x),
\mathrm{e}^{i(\tilde t_{j,n}-\tilde t_{j',n})\Delta}(D^{1/2}h_{j',\lambda_{j',n},x_{j',n}})(x)
\right>.
\end{equation*}
But, this coincides with \eqref{eq321}, which we have already shown goes to 0, concluding the proof of 
\eqref{eq331}.

In order to establish \eqref{eq332}, we consider the first case, $|\tilde t_{j,n}|\leq C_j$, which after passing to a subsequence in $\beta$ , follows from \eqref{eq315} and Lemma \ref{lem36}. The case when 
$\tilde t_{j,n}$ is unbounded follows analogously, using scattering. This finishes the proof of Step 7.
 $\square$ \end{proof}

To conclude the proof of Propostion \ref{pp33}, note that, because of \eqref{eq329} we have
\begin{equation*}
A_{n(m_\beta(k_\alpha),k_\alpha)}^2\geq A_{k_\alpha}^2+\epsilon_{k_\alpha}(\beta).
\end{equation*}
Letting $\beta\to\infty$ we see that $A_C^2\geq A_{k_\alpha}^2$. Letting $\alpha\to\infty$ we obtain
$A_C^2\geq A^2\geq A_C^2$, so that $A^2=A_C^2$ and $U_{j_0}$ is our critical element (see \eqref{eq328}).
 $\square$ \end{proof}

\begin{rem}\label{rem334}
The above proof shows that, for $j\neq j_{0}$, we must have $V_{0,j}=0$ and that $w_n^J\to0$ in $\hdt$. Indeed, let $\epsilon>0$ be given, pick $J>j$. We showed that $A_{k_\alpha}^2\xrightarrow[k_\alpha]{}A_C^2$. Pick 
$k_\alpha$ so large that $|A_{k_\alpha}^2-A_C^2|\leq\epsilon/2$. For this fixed $k_\alpha$, the argument shows that
\begin{equation*}
A_{n(m_\beta(k_\alpha),k_\alpha)}^2\geq A_C^2+\left\lVert\tilde U_{j,n(m_\beta(k_\alpha),k_\alpha)}\left(\tau_{j_0,k_\alpha}^{n(m_\beta(k_\alpha),k_\alpha)}\right)\right\rVert_\hdt^2
-\frac{\epsilon}{2}+\epsilon_{k_\alpha}(\beta)
\end{equation*}
Take now $\beta\to \infty$. We obtain that, for each $\epsilon>0$, there exist $\overline\alpha$,
$\overline\beta$ so that
\begin{equation*}
\left\lVert\tilde U_{j,n(m_{\overline\beta}(k_{\overline\alpha}),k_{\overline\alpha})}\left(\tau_{j_0,k_{\overline\alpha}}^{n(m_{\overline\beta}(k_{\overline\alpha}),k_{\overline\alpha})}\right)\right\rVert_\hdt^2\leq\epsilon.
\end{equation*}
But then, by Theorem \ref{thm24}, $\sup_{t\in(-\infty,+\infty)}||U_j(t)||_\hdt\leq C\epsilon$, so that
$U_j\equiv0$ and hence $V_{0,j}\equiv0$. The argument for $w_n^J$ is similar, using the preservation of the 
$\hdt$ norm by the linear flow.
\end{rem}

\begin{proof}[Proof of Proposition \ref{pp34}]
It follows from the argument in Proposition 4.2 of \cite{KM}, using Remark \ref{rem334}.
\begin{rem}\label{rem335}
Because of the continuity of $u(t)$, $t\in[0,T_+(u_0))$ in $\hdt$, in Proposition \ref{pp34} we can construct 
$\lambda(t)$, $x(t)$ continuous in $[0,T_+(u_0))$, with $\lambda(t)>0$ for each $t\in[0,T_+(u_0))$. (See the proof in Remark 5.4 of \cite{KM}).
\end{rem}
\begin{lem}\label{lem336}
Let $u$ be a critical element as in Proposition \ref{pp34}. Then there is a (possibly different) element $w$, with a corresponding $\tilde\lambda$, and $M_0>0$, so that $\tilde\lambda(t)\leq M_0$ for $t\in[0,T_+(w_0))$,
$||w||_{S([0,T_+(w_0))}=+\infty$, $$\sup_{t\in[0,T_+(w_0))}||w(t)||_\hdt<\infty.$$
\end{lem}
\begin{proof}
(This type of proof originates in \cite{M}. See also \cite{KM}, page 670, for a similar proof). Because of Remark 
\ref{rem335}, we can assume that there exist $\{t_n\}_{n=1}^\infty$, $t_n\geq0$, $t_n\uparrow T_+(u_0)$ so that
\begin{equation*}
\lambda(t_n)\uparrow+\infty.
\end{equation*}
After possibly redefining $\{t_n\}$ we can assume that
\begin{equation*}
\lambda(t_n)\geq\max_{t\in[0,t_n]}\lambda(t).
\end{equation*}
From our hypothesis,
\begin{equation*}
\frac{1}{\lambda(t_n)}u\left(\frac{x-x(t_n)}{\lambda(t_n)},t_n\right)=w_{0,n}(x)\to w_0(x)
\end{equation*}
in $\hdt$. Since $A_C\geq A_0$, by Theorem \ref{thm24} we have $w_0\not\equiv0$. We now consider solutions of 
\eqref{cp}, $w_n(x,\tau)$, $w(x,\tau)$ with data $w_{0,n}$, $w_0$ at $\tau=0$, defined in maximal intervals
$\tau\in(-T_-(w_{0,n}),0]$, $\tau\in(-T_-(w_{0}),0]$. since $w_{0,n}\to w_0$ in $\hdt$, 
$\varliminf T_-(w_{0,n})\geq T_-(w_0)$ and for each $\tau\in(-T_-(w_{0}),0]$, $w_n(x,\tau)\to w(x,\tau)$ in 
$\hdt$. (See Remark \ref{rem213}.) Moreover, by uniqueness in \eqref{cp}, for $0\leq t_n+\tau/\lambda(t_n)^2$, we have 
\begin{equation*}
w_n(x,\tau)=\frac{1}{\lambda(t_n)}u\left(\frac{x-x(t_n)}{\lambda(t_n)},t_n+\frac{\tau}{\lambda(t_n)^2}\right).
\end{equation*}
Let $\tau_n$ be defined by $t_n+\tau_n/\lambda(t_n)^2=0$. Note that
$\varliminf-\tau_n=\varliminf t_n\lambda(t_n)^2$ $\geq T_-(w_0)$. If not 
\begin{equation*}
w_n(w,\tau_n)=\frac{1}{\lambda(t_n)}u_0\left(\frac{x-x(t_n)}{\lambda(t_n)}\right)\to w(x,\tau_0),\quad
\tau_0\in(-T_-(w_0),0],
\end{equation*}
in $\hdt$, which is a contradiction to $\lambda(t_n)\uparrow+\infty$, $w_0\not\equiv0$. Thus, for all
$\tau\in(-T_-(w_{0}),0]$, for $n$ large, $0\leq t_n+\tau/\lambda(t_n)^2\leq t_n$.

Fix now $\tau\in(-T_-(w_{0}),0]$ and let $v(x,t)$ be as in Proposition \ref{pp34}. For $n$ sufficiently large, 
$\lambda(t_n+\tau/\lambda(t_n)^2)$ and $v(x,t_n+\tau/\lambda(t_n)^2)$ are defined and we have
\begin{multline}\label{eq337}
v\left(x,t_n+\frac{\tau}{\lambda(t_n)^2}\right)=\\\frac{1}{\lambda(t_n+\tau/\lambda(t_n)^2)}
u\left(\frac{x-x(t_n+\tau/\lambda(t_n)^2)}{\lambda(t_n+\tau/\lambda(t_n)^2)},
t_n+\tau/\lambda(t_n)^2\right)=\\
\frac{1}{\tilde\lambda_n(\tau)}w_n\left(\frac{x-\tilde x_n(\tau)}{\tilde\lambda_n(\tau)},\tau\right),
\end{multline}
where
\begin{equation*}
\tilde\lambda_n(\tau)=\frac{\lambda(t_n+\tau/\lambda(t_n)^2)}{\lambda(t_n)},
\end{equation*}
\begin{equation*}
\tilde x_n(\tau)=x(t_n+\tau/\lambda(t_n)^2)-x(t_n)/\tilde\lambda(t_n).
\end{equation*}

Note that $0<\tilde\lambda_n(\tau)\leq1$. Note also that $||w_n(\cdot,\tau)||_\hdt\leq A_C$, for each $\tau$, so that 
\begin{equation*}
\sup_{t\in(-T_0(w_0),0]} ||w(\tau)||_\hdt\leq A_C. 
\end{equation*}
Note also that $||w||_{S(-T_-(w_0),0)}=\infty$. Otherwise, $T_-(w_{0})=+\infty$ and 
by Theorem \ref{thm212}, for $n$ large, $T_-(w_{0,n})=+\infty$ and 
$||w_n||_{S(-\infty,0)}\leq M$, which contradicts 
$||u||_{S(0,T_+(u_0))}=+\infty$. Finally, since 
\begin{equation*}
\frac{1}{\lambda_n^{3/2}}h((x-x_n)/\lambda_n)\xrightarrow[n\to\infty]{}h_0
\end{equation*}
in $L^2$, with either 
$\lambda_n\to 0 \text{ or } \infty$ or $|x_n|\to\infty$, implies that $h_0\equiv0$ and since no element in 
$\overline K$ can be zero by $A_C\geq A_0>0$ and uniqueness in \eqref{cp}, we can assume, after passing to a subsequence that $\tilde\lambda_n(\tau)\to\tilde\lambda(\tau)$, $0<\tilde\lambda(\tau)\leq1$, 
$\tilde x_n(\tau)\to\tilde x(\tau)\in\R^3$. But then
\begin{equation*}
\frac{1}{\tilde\lambda(\tau)}w\left(\frac{x-\tilde x(\tau)}{\tilde\lambda(\tau)},\tau\right)\in\overline K
\end{equation*}
as desired. (Actually we should take $\overline w(x,-\tau)$ as our new  critical element.)
 $\square$ \end{proof}

\end{proof}

\section{Rigidity Theorem}

In this section we will prove the following:
\begin{thm}\label{thm41}
Assume that $u_0\in\hdt$ is such that, for $u$ the solution of \eqref{cp} with maximal interval $[0,T_+(u_0))$, we have the following properties:
\begin{itemize}
\item[i)] $\displaystyle\sup_{0\leq t<T_+(u_0)}||u(t)||_\hdt\leq A$
\item[ii)] $\displaystyle||u||_{S(0,T_+(u_0))}=+\infty$
\item[iii)] There exist continuous functions $\lambda(t)$, $x(t)$ in $[0,T_+(u_0))$, with $0<\lambda(t)\leq M_0$, $t\in[0,T_+(u_0))$, so that
\begin{equation*}
K=\left\{v(x,t)=\frac{1}{\lambda(t)}u\left(\frac{x-x(t)}{\lambda(t)},t\right)\right\}
\end{equation*}
has compact closure in $\hdt$.
\end{itemize}
Then no such $u_0$ exists.
\end{thm}

For the proof of Theorem \ref{thm41}, note that, by translation and scaling, we can assume $x(0)=0$, $\lambda(0)=1$. Moreover, in light of ii) and Theorem \ref{thm24}, we can assume that $||u(t)||_\hdt\geq A_0>0$ for each $t\in[0,T_+(u_0))$. From now on we consider such a $u$. We need some lemmas in order to carry out the proof of Theorem \ref{thm41}.

\begin{lem}\label{lem42}
Let $t_0\in[0,T_+(u_0))$, 
\begin{equation*}
w_0=\frac{1}{\lambda(t_0)}u\left(\frac{x-x(t_0)}{\lambda(t_0)},t_0\right)\in K,
\end{equation*}
 and $w(x,t)$ be the solution of 
\eqref{cp} with data $w_0$. Then there exist $\tau_0=\tau_0(K)>0$, $\alpha_0=\alpha_0(K)>0$, $R_0=R_0(K)>0$, $M_1=M_1(K)>0$, so that,
\begin{itemize}
\item[i)] $w(t)$ is defined for $[0,2\tau_0]$
\item[ii)] $\forall t\in[0,\tau_0]$, we have
\begin{equation*}
||w(t)||_{L^3(|x|\leq R_0)}\geq\alpha_0
\end{equation*}
\item[iii)] $\displaystyle\frac{1}{M_1}\leq\frac{\lambda(t_0+t/\lambda(t_0)^2)}{\lambda(t_0)}\leq M_1$
\item[iv)] $\displaystyle\left|x(t_0+t/\lambda(t_0)^2)-\frac{\lambda(t_0+t/\lambda(t_0)^2)}{\lambda(t_0)}x(t_0)\right|\leq M_1$.
\end{itemize}
\end{lem}
\begin{proof}
Since $w_0\in K$, we can find $\tau_1(K)>0$ so that the family $\{w(x,t)\}$ is defined and equicontinuous in $[0,2\tau_1]$. (See Remark \ref{rem214}.) We next claim that $\exists\;\alpha_0(K), R_0(K)>0$, so that
\begin{equation}\label{eq43}
||w_0||_{L^3(|x|\leq R_0)}\geq2\alpha_0>0.
\end{equation} 
In fact, $\overline K$ is compact in $\hdt(\R^3)$ and hence in $L^3(\R^3)$. If \eqref{eq43} fails, we can find $w_n\in K$, $R_n\to+\infty$, so that $||w_n||_{L^3(|x|\leq R_n)}\to0$. By passing to a subsequence, we can find $v\in\overline K$ so that $w_n\to v$ in $\hdt$. But then
$||v||_{L^3(|x|\leq R)}=0$ for each $R$, so that $v\equiv0$. But $||v||_\hdt\geq A_0>0$, a contradiction. By equicontinuity on 
$\overline K$, we can find $0<\tau_0<\tau_1$, $\tau_0=\tau_0(K)$ so that, for each $t\in[0,\tau_0]$, $||w_0-w(t)||_{L^3}\leq\alpha_0$, so that $||w(t)||_{L^3(|x|\leq R_0)}\geq\alpha_0>0$.

To show iii) and iv), define
\begin{equation*}
\begin{split}
\lambda_{t_0}(t)&=\frac{\lambda(t_0+t/\lambda(t_0)^2)}{\lambda(t_0)}\\
x_{t_0}(t)&=x(t_0+t/\lambda(t_0)^2)-\frac{\lambda(t_0+t/\lambda(t_0)^2)}{\lambda(t_0)}x(t_0).
\end{split}
\end{equation*}
Note that, if $t_0+t/\lambda(t_0)^2<T_+(u_0)$, by uniqueness in \eqref{cp} we have that 
\begin{equation}\label{eq44}
w(x,t)=\frac{1}{\lambda(t_0)}u\left(\frac{x-x(t_0)}{\lambda(t_0)},t_0+t/\lambda(t_0)^2\right).
\end{equation}
This shows that for $t\in[0,2\tau_0]$, $t_0+t/\lambda(t_0)^2<T_+(u_0)$. Moreover, from \eqref{eq44} we see that
\begin{multline*}
\frac{1}{\lambda_{t_0}(t)}w\left(\frac{x-x_{t_0}(t)}{\lambda_{t_0}(t)},t\right)=\\=
\frac{1}{\lambda(t_0+t/\lambda(t_0)^2)}u\left(\frac{x-x(t_0+t/\lambda(t_0)^2)}{\lambda(t_0+t/\lambda(t_0)^2)},t_0+t/\lambda(t_0)^2\right)
\in K.
\end{multline*}
To conclude that $\exists M_1$ so that $\forall t\in[0,\tau_0]$, $\forall t_0\in[0,T_+(u_0))$, we have 
$\frac{1}{M_1}\leq\lambda_{t_0}(t)\leq M_1$, $|x_{t_0}(t)|\leq M$, assume not. Then there is a sequence $w_{0;t_{0,n}}\in K$, with corresponding solution  $w_n$, so that 
\begin{equation*}
\frac{1}{\lambda_n}w_n\left(\frac{x-x_n}{\lambda_n},t_n\right)\in K
\end{equation*}
and
$\lambda_n+{1}/{\lambda_n}+|x_n|\to\infty$, where $\lambda_n=\lambda_{t_{0,n}}(t_n)$, $x_n=x_{t_{0,n}}(t_n)$. After taking a subsequence, we can assume $t_n\to\overline t\in[0,2\tau_0]$, $w_{0;t_{0,n}}\to v_0\in\overline K$, 
\begin{equation*}
\frac{1}{\lambda_n}w_n\left(\frac{x-x_n}{\lambda_n},t_n\right)\to v_1\in \overline K.
\end{equation*}
 Since $w_n(t_n)\to v(\overline t)$, where $v$ is the solution corresponding to $v_0$, we see that 
\begin{equation*}
\frac{1}{\lambda_n}v\left(\frac{x-x_n}{\lambda_n},\overline t\right)\to v_1.
\end{equation*}
 But, since 
$\lambda_n+{1}/{\lambda_n}+|x_n|\to\infty$, $v_1\equiv0$, which is a contradiction, since $||v||_\hdt\geq A_0>0$ for all 
$v\in\overline K$. This concludes the proof of Lemma \ref{lem42}.
 $\square$ \end{proof}

\begin{cor}\label{cor45}
For all $t_0\in[0,T_+(u_0))$ we have
\begin{itemize}
\item[i)] $\displaystyle\frac{2\tau_0}{\lambda(t_0)^2}\leq T_+(u_0)-t_0$
\item[ii)] $\displaystyle\int_{t_0}^{t_0+\tau_0/\lambda(t_0)^2}\int\frac{|u|^4}{|x|}dxdt\geq
\frac{C_0(K)}{|x(t_0)|+R_0}$,
\end{itemize}
where $C_0(K)>0$, $R_0$ is as in Lemma \ref{lem42}.
\end{cor}
\begin{proof}
i) was observed after \eqref{eq44}. For ii), change variables to see that the integral in ii) equals
\begin{equation*}
\int_0^{\tau_0}\int\frac{|u(x,t_0+t/\lambda(t_0)^2)|^4}{|x|}dx\frac{dt}{\lambda(t_0)^2}.
\end{equation*}
But \eqref{eq44} gives that the integral equals
\begin{multline*}
\int_0^{\tau_0}\int\lambda(t_0)^2\frac{|w(\lambda(t_0)x+x(t_0),t)|^4}{|x|}dxdt=\\=
\int_0^{\tau_0}\int\frac{|w(y+x(t_0),t)|^4}{|y|}dydt\geq\\\geq
\frac{1}{|x(t_0)|+R_0}\int_0^{\tau_0}\int_{|y+x(t_0)|\leq R_0}|w(y+x(t_0),t)|^4dydt\geq\\\geq
\frac{C_{R_0}\tau_0}{|x(t_0)|+R_0}
\end{multline*}
by ii) in Lemma \ref{lem42}.
 $\square$ \end{proof}

Let us now define $t_0=0$, $t_{n+1}=t_n+\tau_0/\lambda(t_n)^2$. Note that $0\leq t_n<T_+(u_0)$, by i) in Corollary \ref{cor45}. Moreover $t_n<t_{n+1}$. Let now
\begin{equation}\label{eq46}
I_n=\int_{t_n}^{t_{n+1}}\int\frac{|u|^4}{|x|^4}
\end{equation}
where $u$ is as in Theorem \ref{thm41}.

\begin{lem}\label{lem47}
Let $\alpha_n=1/\lambda(t_n)$, where $t_n$, $I_n$ are as above. Then
\begin{equation}\label{eq48}
\frac{\alpha_n}{\sum_{j=1}^n\alpha_j}\leq CI_n,
\end{equation}
for some fixed constant $C=C(K)$.
\end{lem}
\begin{proof}
From Corollary \ref{cor45}, ii),
\begin{equation*}
\frac{C_0(K)}{|x(t_n)|+R_0}\leq I_n.
\end{equation*}
Moreover, from Lemma \ref{lem42}, iii), iv), we have
\begin{equation*}
\frac{1}{M_1}\leq\frac{\lambda(t_{n+1})}{\lambda(t_n)}\leq M_1
\end{equation*}
and
\begin{equation*}
\left|\frac{x(t_{n+1})}{\lambda(t_{n+1})}-\frac{x(t_{n})}{\lambda(t_{n})}\right|\leq\frac{M_1}{\lambda(t_{n+1})}.
\end{equation*}
But then 
\begin{equation*}
\left|\frac{x(t_{n})}{\lambda(t_{n})}\right|\leq M_1\sum_{j=1}^n\alpha_j
\end{equation*}
and
\begin{multline*}
\frac{1}{|x(t_n)|+R_0}=
\frac{1}{\lambda(t_n)\left[\frac{|x(t_n)|}{\lambda(t_n)}+\frac{R_0}{\lambda(t_n)}\right]}\geq\\\geq
\frac{1}{\lambda(t_n)}\frac{1}{M_1(\sum_{j=1}^n\alpha_j)+R_0\alpha_n}\geq
C\frac{\alpha_n}{\sum_{j=1}^n\alpha_j},
\end{multline*}
as desired.
 $\square$ \end{proof}

\begin{lem}\label{lem49}
Let $\alpha_n$ be a sequence of non-negative real numbers, with $\alpha_n\geq1/M_0$ and 
$\frac{1}{M_1}\leq\frac{\alpha_{n+1}}{\alpha_n}\leq M_1$. Then, if $s_n=\frac{\alpha_n}{\sum_{j=1}^n\alpha_j}$, we have $\sum_{n=1}^\infty s_n=+\infty$.
\end{lem}
\begin{proof}
For $1\leq r<\infty$, let $g(r)=\alpha_n$, if $r\in[n,n+1)$. Define $G(r)=1+\int_1^rg(s)ds$. Note that 
$G\uparrow$ and that $G(n+1)\geq1+\sum_{j=1}^n\alpha_j$. Since $\alpha_n\geq1/M_0$, $G(r)\uparrow+\infty$
as $r\to+\infty$. Also, $G'(r)=g(r)$, hence, $\int_1^\infty g(s)/G(s) ds=+\infty$. But,
\begin{equation*}
\int_1^{n+1}\frac{g(s)}{G(s)}ds\leq\sum_{j=1}^n\int_j^{j+1}\frac{g(s)}{G(s)}ds=
\sum_{j=1}^n\frac{\alpha_j}{1+\sum_{l=1}^{j-1}\alpha_l}.
\end{equation*}
but, $\alpha_{j-1}=\frac{1}{2}\alpha_{j-1}+\frac{1}{2}\alpha_{j-1}\geq\frac{1}{2}\alpha_{j-1}+
\frac{1}{2M_1}\alpha_{j}$, so
\begin{equation*}
\int_1^{n+1}\frac{g(s)}{G(s)}ds\leq C\sum_{j=1}^n\frac{\alpha_j}{\sum_{l=1}^{j}\alpha_l},
\end{equation*}
and the Lemma follows.
 $\square$ \end{proof}

\begin{pp}\label{pp410}
Let $u$ be as in Theorem \ref{thm41}. Then,
\begin{equation*}
\int_0^{T_+(u_0)}\int\frac{|u(x,t)|^4}{|x|}dxdt=+\infty.
\end{equation*}
\end{pp}
\begin{proof}
\begin{multline*}
\int_0^{T_+(u_0)}\int\frac{|u(x,t)|^4}{|x|}dxdt\geq\\\geq
\sum_{n=1}^{+\infty}\int_{t_n}^{t_{n+1}}\int\frac{|u(x,t)|^4}{|x|}dxdt=
\sum_{n=1}^\infty I_n\geq C\sum_{n=1}^\infty\frac{\alpha_n}{\sum_{j=1}^n\alpha_j}=+\infty,
\end{multline*}
in light of \eqref{eq48}, the fact that $\lambda(t)\leq M_0$ and that 
$\frac{1}{M_1}\leq\frac{\alpha_{n+1}}{\alpha_n}\leq M_1$ (Corollary \ref{cor45} ii), combined with Lemma 
\ref{lem49}).
 $\square$ \end{proof}

We now turn to some upper bounds, which are consequences of the Morawetz type identity of Lin--Strauss \cite{LS}.

\begin{lem}\label{lem411}
Let $u_0\in H^1\cap\hdt$. Then, for each $0<T<T_+(u_0)$, we have 
\begin{equation*}
\int_0^T\int\frac{|u(x)|^4}{|x|}dxdt\leq C_0\left[||u(T)||^2_\hdt+||u(0)||^2_\hdt\right]
\end{equation*}
where $u$ is the solution of \eqref{cp}, $C_0$ is independent of $T$.
\end{lem}
\begin{proof}
See for instance Proposition 2.1 and Lemma 2.1 in \cite{CKSTT}.
 $\square$ \end{proof}

From Lemma \ref{lem411} we imediately obtain:
\begin{lem}\label{lem412}
Let $u$ be a solution of \eqref{cp}, so that $u\in B(A)$ (see Definition \ref{dfn31}). Then,
\begin{equation*}
\int_0^{T_+(u_0)}\int\frac{|u(x)|^4}{|x|}dxdt\leq 2C_0A^2.
\end{equation*}
\end{lem}
\begin{proof}
Fix $T<T_+(u_0)$. Pick $u_{0.n}\in H^1\cap\hdt$, so that $u_{0,n}\to u_0$ in $\hdt$. By Remark \ref{rem213}, for each $0\leq t\leq T$, we have $u_n(t)\to u(t)$ in $\hdt$, where $u_n$ is the solution of \eqref{cp} corresponding to $u_{0,n}$. But,
\begin{multline*}
\int_0^T\int\frac{|u(x)|^4}{|x|}dxdt\leq
\varliminf_{n\to\infty}\int_0^T\int\frac{|u_n(x)|^4}{|x|}dxdt\leq\\ \\\leq
\varliminf_{n\to\infty} C_0\left[||u_n(T)||^2_\hdt+||u_n(0)||^2_\hdt\right]\leq
2 C_0A^2.
\end{multline*}
 $\square$ \end{proof}

Now, Proposition \ref{pp410} and Lemma \ref{lem412} immediately yield Theorem \ref{thm41}.

\section{Conclusion of the proof of Theorem \ref{thm11} and further results}

Note that Lemma \ref{lem336}, combined with Theorem \ref{thm41} immediately yields our main result, Theorem 
\ref{thm11}. We now list some Corollaries.

\begin{cor}\label{cor51}
If $u_0\in\hdt$, $T_+(u_0)<+\infty$, then 
\begin{equation*}
\sup_{t\in[0,T_+(u_0))}||u(t)||_\hdt=+\infty.
\end{equation*}
\end{cor}
This is immediate from Theorem \ref{thm11}.

\begin{cor}\label{cor52}
$B(\infty)$ (see Definition \ref{dfn31}) is open in $\hdt$.
\end{cor}
\begin{proof}
Assume that $u_0\in B(\infty)$, so that for some $A$, $u_0\in B(A)$. In light of Theorem \ref{thm11}, 
$T_+(u_0)=+\infty$ and $||u||_{S(0,+\infty)}<\infty$. But now Theorem \ref{thm212} yields the corollary.
 $\square$ \end{proof}

\begin{cor}\label{cor53}
There exists an increasing function $g(A)$ so that, if $u_0\in B(A)$, we have
\begin{equation*}
||u||_{S(0,\infty)}\leq g(A).
\end{equation*}
\end{cor}
The proof of Corollary \ref{cor53} is similar to the one of \cite{K}, Corollary 1.14, using arguments in Section 3.

\begin{rem}\label{rem54}
For radial data $u_0$, one can give a strenghtening of Theorem \ref{thm11}, namely that the condition 
$\varliminf_{t\uparrow T_+(u_0)}||u(t)||_\hdt<\infty$ suffices to guarantee that $T_+(u_0)=+\infty$ and 
$||u||_{S(0,+\infty)}<\infty$. There are several ways to see that, some along the lines used in our work, but the quickest argument centers on the fact that using the weighted Strichartz estimates in \cite{V} and radial Sobolev embeddings (\cite{V},\cite{SW}) one can see that, for radial data, in Lemma \ref{lem28} and Remark \ref{rem29}, one can replace the $S$-norm for the norm $L^4_IL^4(dx/|x|)$ and then use Lemma \ref{lem412}.
\end{rem}

\end{document}